\documentclass[11pt, epsfig]{article}
\usepackage{epsfig, amsmath, amssymb, amsthm, times}
\usepackage{graphicx}

\newcommand{\R}{{\mathbb R}}
\newcommand{\C}{{\mathbb C}}
\newcommand{\N}{{\mathbb N}}
\newcommand{\Z}{{\mathbb Z}}
\newcommand{\Rep}{\operatorname{Re}}

\def\bT{{\mathbf T}}
\def\bP{{\mathbf P}}
\def\bW{{\mathbf W}}
\def\bR{{\mathbf R}}
\def\bI{{\mathbf I}}
\def\bD {\operatorname{\mathbf D}}
\def\bPi{\operatorname{\mathbf \Pi}}

\usepackage[mathcal]{eucal}

\usepackage[pagewise]{lineno}\nolinenumbers

\newtheorem {thm}{Theorem}[section]
\newtheorem {lem}[thm]{Lemma}

\newtheorem {cor}[thm]{Corollary}

\theoremstyle{defintion}
\newtheorem {df}[thm]{Definition}

\theoremstyle{remark}
\newtheorem{rem}[thm]{Remark}

\theoremstyle{example}

\theoremstyle{assumption}

\def\SS{{\mathbb S}}

\def\lbl{\label}
\def\be{\begin{equation}}
\def\ee{\end{equation}}
\def\p{\partial}

\def\wlim{\mathop{\operatorname{lim*}}_{\lambda\to 0+}}
\def\l{\lim_{\lambda\to 0+}}

\newcommand{\1}{{i\mkern1mu}}

\title{The mean field analysis of the Kuramoto model on graphs II.\\
Asymptotic stability of the incoherent state, center manifold reduction,
and bifurcations}
\author{Hayato Chiba\thanks{Institute of Mathematics for Industry, 
Kyushu University / JST PRESTO, Fukuoka,
819-0395, Japan, {\tt chiba@imi.kyushu-u.ac.jp}}\;
and Georgi S. Medvedev\thanks{Department of Mathematics, 
Drexel University, 3141 Chestnut Street, Philadelphia, PA 19104,
{\tt medvedev@drexel.edu}} 
}
\begin{document}
\maketitle

\begin{abstract}
In our previous work \cite{ChiMed16},  we initiated a mathematical investigation of the
onset of synchronization in the Kuramoto model (KM) of coupled phase oscillators
on convergent graph sequences. There, we derived and rigorously justified
the mean field limit for the KM on graphs. Using linear stability analysis, we identified
the critical values of the coupling strength, at which the incoherent state looses stability,
thus, determining the onset of synchronization in this model.

In the present paper, we study the corresponding bifurcations. Specifically, we show that similar
to the original KM with all-to-all coupling, the onset of synchronization in the KM on graphs
is realized via a pitchfork bifurcation. The formula for the stable branch of the 
bifurcating equilibria involves the principal eigenvalue and the corresponding eigenfunctions
of the kernel operator defined by the limit of the graph sequence used in the model. This establishes
an explicit link between the network structure and the onset of synchronization in the KM on graphs.
The results of this work are illustrated with the bifurcation analysis of the KM on Erd\H{o}s-R{\' e}nyi,
small-world, as well as certain weighted graphs on a circle. 
\end{abstract}

\section{Introduction}
\setcounter{equation}{0}

In 1970s, a prominent Japanese physicist Yoshiki Kuramoto described 
a remarkable effect in collective dynamics of large systems of coupled 
oscillators \cite{Kur75}.  He studied all-to-all coupled phase oscillators
with randomly distributed intrinsic frequencies, the model which now 
bears his name. When the strength of coupling is small, the phases are 
distributed approximately uniformly around a unit circle, forming an 
incoherent state. Kuramoto identified a critical value of the coupling 
strength, at which the the incoherent state looses stability giving
rise to a stable partially synchronized state. To describe the bifurcation
corresponding to the onset of synchronization, Kuramoto introduced the 
order parameter - a scalar function, which measures the degree of coherence
in the system. He further showed that the order parameter undergoes 
a pitchfork bifurcation. Thus, the qualitative changes in the statistical behavior
of a large system of coupled phase oscillators near the onset to synchronization
can be described in the language of the bifurcation theory.

Kuramoto's discovery created a new area of research in nonlinear science
\cite{StrMir91, StrMir92, Str00}. The rigorous mathematical treatment 
of the pitchfork bifurcation in the KM was outlined in \cite{ChiNis2011} and
was presented with all details in \cite{Chi15a}. The analysis in these papers
is based on the generalized spectral theory for linear operators 
\cite{Chi15b} and applies to the KM with intrinsic frequencies
drawn from a distribution with analytic or rational probability density function.
Under more general assumptions on the density, the onset of synchronization
in the KM was analyzed in \cite{Die16,FGG16,Die17}. These papers use 
analytical methods for partial differential equations and build upon a 
recent breakthrough in the analysis of Landau damping \cite{MouVil2011}.

In our previous work \cite{ChiMed16},  we initiated a mathematical study of the 
onset of synchronization in the KM on graphs. Following \cite{Med14a, Med14b},
we considered the KM on convergent families of deterministic and random
graphs, including Erd\H{o}s-R{\' e}nyi and small-world graphs among many other 
graphs that come up
in applications. For this model, we derived and rigorously justified the mean field
limit. The latter is  a partial differential equation approximating the dynamics of the coupled oscillator
model in the continuum limit as the number of oscillators grows to infinity. In \cite{ChiMed16},
we performed a linear stability analysis of the incoherent state and identified the 
boundaries of stability. Importantly, we related the stability region of the
incoherent state to the structural properties of the network through the spectral properties
of the kernel operator defined by the limit of the underlying graph sequence 
\cite{LovGraphLim12}. In the present paper, we continue the study initiated in \cite{ChiMed16}.
Here, we analyze the bifurcations at the critical values of the coupling strength $K_c^-<0<K_c^+$, where 
the incoherent state looses stability. As in the analysis of the original KM in \cite{Chi15a},
we have to deal with the fact that for $K\in [K_c^-, K_c^+]$ the linearized operator
has continuous spectrum on the imaginary axis and no eigenvalues. Thus, neither the asymptotic stability 
of the incoherent state, nor the center manifold reduction for $K=K^\pm_c$ can be performed 
using standard methods of the bifurcation theory. To overcome this problem, following \cite{Chi15a}, 
we develop the generalized spectral theory, which allows effective analysis of the bifurcations at $K_c^\pm$.
This involves generalizing the resolvent operator and using it to  define generalized eigenvalues.
The generalized spectral theory is used i) to prove asymptotic stability of the incoherent state 
as a steady state solution of the linearized equation for $K\in [K_c^-,K_c^+]$; ii) to obtain finite-dimensional
center manifolds at the critical values of the coupling strength $K=K_c^\pm$; iii) to identify the 
bifurcations at $K=K_c^\pm$. We do not develop nonlinear estimates in this paper. We also do not 
prove the existence of the center manifold rigorously. Both problems can be dealt with following the 
lines of the analysis in \cite{Chi15a}. These very technical questions are left out to keep the length of this 
paper reasonable. Instead, we focus on the analysis of the bifurcations at $K^\pm_c$. To this end, we develop 
all necessary tools for setting up the bifurcation problems, which are analyzed formally assuming the existence
of the center manifold. The center manifold reduction yields stable spatial patterns arising from the bifurcations
at the critical values of the coupling strength. We apply these results to the KM on  several representative graphs,
including small-world and Erd\H{o}s-R{\' e}nyi graphs. Our results for the KM on these and many other graphs
are verified numerically in the follow-up paper \cite{ChiMedMiz18}.
In the remainder of this section, we review the 
model and the main outcomes of \cite{ChiMed16} and explain the main results of this 
work.

We begin with a brief explanation of the graph model that was used in \cite{ChiMed16}
and will be used in this paper. In \cite{ChiMed16}, we adapted the construction of
W-random graphs \cite{LovSze06} to define a convergent sequence of graphs.
The flexible framework of W-random graphs allows us to deal with a broad class
of networks that are of interest in applications.
Let $W$ be a symmetric measurable function on the unit square $I^2:=[0,1]^2$
with values in $[-1,1]$ and let
$
X_n=\{\xi_{n1}, \xi_{n2},\dots,\xi_{nn}\}
$
form a triangular array of points $\xi_{ni},\; i\in [n]:=\{1,2,\dots n\}, \; n\in \N,$ 
subject to the following condition
\be\lbl{wXn}
\lim_{n\to\infty} n^{-1}\sum_{i=1}^n f(\xi_{ni})= \int_I f(x)dx \quad
\forall f\in C(I).
\ee
$\Gamma_n=\langle V(\Gamma_n), E(\Gamma_n), (W_{nij})\rangle$ 
is a weighted graph on $n$ nodes labeled by the integers from $[n]$, whose edge set 
is 
$$
E(\Gamma_n)=\left\{ \{i,j\}:\;  W(\xi_{ni},\xi_{nj})\neq 0,\; i,j\in [n]\right\}.
$$
Each edge $\{i,j\}\in E(\Gamma_n)$ is supplied with the weight $W_{nij}:=W(\xi_{ni}, \xi_{nj}).$
In the theory of graph limits, $W$ is called a graphon \cite{LovGraphLim12}. 
It defines the asymptotic properties of $\{\Gamma_n\}$ for large $n$.

Consider the initial value problem (IVP) for the  KM on $\Gamma_n$ 
\begin{eqnarray}
\lbl{KM}
\dot \theta_{ni} &=& \omega_{i} + Kn^{-1} \sum_{j=1}^n W_{nij} \sin(\theta_{nj}-\theta_{ni}),
\quad i\in [n],\\
\lbl{KM-ic}
\theta_{ni}(0)&=&\theta_{ni}^0. 
\end{eqnarray}
The intrinsic frequencies $\omega_i, \; i\in [n],$ are independent identically distributed 
random variables. The distribution of $\omega_1$ has density $g(\omega)$. 
For the spectral analysis in Section~\ref{sec.spectral}, we need to impose the 
following assumptions on $g$: a) $g:\R\to\R^+\cup \{0\}$ 
is an even unimodal function, and b) $g$ is real analytic function with finite 
moments of all orders: $\int_\R |x|^m g(x)dx <\infty,$ $m\in \N$. For instance,
the density of the Gaussian distribution satisfies these conditions.  
The KM on weighted graphs $\{\Gamma_n\}$ \eqref{KM}, \eqref{KM-ic} can be used
to approximate
the KM on a variety of random graphs (cf. \;\S~4.2 \cite{ChiMed16}).

Along with the discrete model \eqref{KM} we consider the IVP for the 
following partial differential equation
\begin{eqnarray}
\lbl{MF}
{\p \over \p t} \rho(t,\theta,\omega,x) +
{\p \over \p \theta }\left\{\rho(t,\theta,\omega,x) V(t,\theta,\omega,x)\right\}
&=&0,\\
\lbl{MF-ic}
\rho(0,\theta,\omega, x) &=& \rho^0(\theta,\omega,x)\in \SS\times\R\times I,
\end{eqnarray}
where
\be\lbl{def-V}
V(t,\theta,\omega,x)=\omega+
K\int_I\int_\R\int_\SS W(x,y) \sin(\phi - \theta ) \rho(t,\phi, \omega ,y) g(\omega ) 
d\phi d\omega d y.
\ee
Here, $\rho(t,\theta,\omega,x)$ is the conditional density of the random vector $(\theta,\omega)$
given $\omega$, and parametrized by $(t,x)\in \R^+\times I$, and $\SS = \R/2\pi \Z$ is a circle.
In particular,
\be\lbl{normalize-g}
\int_\SS \rho(t,\theta,\omega, x) d\theta =1 \quad \forall (t,\omega,x)\in \R^+\times \R\times I.
\ee
It is shown in \cite[Theorem~2.2]{ChiMed16} that 
\be\lbl{EMeasure}
\mu_t^n(A)=n^{-1} \sum_{i=1}^n \delta_{(\theta_{ni}(t), \omega_i, \xi_{ni})} (A) 
\ee
interpreted as a probability measure on Borel sets $A\in\mathcal{B}(G),$
$G=\SS\times\R\times I,$ converges in the bounded Lipschitz distance \cite{Dud02}
uniformly on bounded time intervals to the absolutely continuous measure
\be\lbl{CMeasure}
\mu_t(A)=\int_A \rho(t,\theta,\omega, x)g(\omega ) d\theta d\omega dx, \quad A\in\mathcal{B}(G),
\ee
provided $\mu_0^n$ and $\mu_0$ are sufficiently close in the same distance.
The latter can be achieved with the appropriate initial condition \eqref{KM-ic} and sufficiently
large $n$ (see \cite[Corollary~2.3]{ChiMed16}).
Therefore, the IVP \eqref{MF},\eqref{MF-ic} approximates the IVP \eqref{KM},\eqref{KM-ic}
on finite time intervals for sufficiently large $n$.

An inspection of \eqref{MF} shows that $\rho_u=1/(2\pi)$, the density of the uniform distribution on $\SS$,
is a steady state solution of \eqref{MF}. It corresponds to the incoherent (mixing) state 
of the KM. Numerics suggests that the incoherent state is stable for small $K\ge 0.$ The loss
of stability of the incoherent state is interpreted as the onset of synchronization in the KM.
This is the main focus of \cite{ChiMed16} and of the present paper. In \cite{ChiMed16}, we identified
the boundaries of the region of stability of the incoherent state in \eqref{MF}. Specifically,
we showed that there exist $K_c^-\le 0\le K_c^+$ such that $\rho_u$ is linearly stable
for $K\in [K^-_c, K^+_c]$, and is unstable otherwise. 

The critical values $K_c^-$ and $K_c^+$ depend on the network topology through the 
eigenvalues of the compact symmetric operator $\mathbf{W}:L^2(I)\to L^2(I)$ 
\be\lbl{def-W}
\bW[f](x)=\int_\R W(x,y) f(y) dy.
\ee
The eigenvalues of $\bW$ are real with the only accumulation point at $0$.
Denote the largest positive and smallest negative eigenvalues of $\bW$
by $\mu_{max}$ and $\mu_{min}$ respectively.
If all eigenvalues are nonnegative (nonpositive), we set $\mu_{min}=-\infty$ 
($\mu_{max}=\infty$). The main stability result of \cite{ChiMed16} yields
explicit expressions for the transition points
\be\lbl{wKc}
K_c^-={2\over \pi g(0) \mu_{min}}
\quad\mbox{and}\quad
K_c^+={2\over \pi g(0) \mu_{max}}.
\ee
Thus, the region of linear stability of $\rho_u$ depends explicitly on the spectral properties 
of the limiting graphon $W$. Recall that $W$ represents the graph limit of $\{\Gamma_n\}$.
Thus, \eqref{wKc}  links network topology to synchronization in \eqref{KM}.
For the classical KM (all-to-all coupling), $W(x,y) = 1$ and $\mu_{min} = -\infty,\, \mu_{max} = 1$,
which recovers the known Kuramoto's formula.

In the present paper, we study the onset of synchronization in \eqref{KM} in more detail.
After some preliminaries and preparatory work in Sections \ref{sec.prelim} and \ref{sec.spectral},
we revisit linear stability of  the incoherent solution. This time, we show that despite
the lack of eigenvalues with negative real part and the presence of the continuous spectrum
on the imaginary axis, the incoherent state is an asymptotically stable solution of the 
linearized problem (cf. Theorem~\ref{thm.asympt}). This is a manifestation of the Landau 
damping in the KM.

In Section \ref{sec.bif}, we study the bifurcation at $K_c^+$ with a one-dimensional 
center manifold.
To this end, we recall the order parameter
\be\lbl{Worder}
h(t,x) =\int_I\int_\R\int_\SS W(x,y) e^{\1 \theta} \rho(t,\theta,\omega,y) g(\omega) 
d\theta d\omega dy,
\ee
which was introduced in \cite{ChiMed16} as a measure of coherence in the KM on graphs.
This is a continuous analog of the local order parameter
\begin{eqnarray*}
\frac{1}{n}\sum^n_{j=1}W_{nij}e^{i \theta _{nj}(t)}
\end{eqnarray*}
for the discrete model (\ref{KM}).
The order parameter generalizes the original order parameter used by Kuramoto for the 
all-to-all coupled model. Note that \eqref{Worder} depends on $x$ and contains information
about the structure of the network through $W$. As will be clear below, the order parameter 
plays an important role in the analysis of the mean field equation. In particular, it can be 
used to locate nontrivial steady state solutions. To this end note that the velocity field
\eqref{def-V} can be conveniently rewritten in terms of the order parameter 
\be\lbl{re-V}
V(t,\theta,\omega,x)=\omega+{K\over 2\1} \left(e^{-\1\theta} h(t,x)+e^{\1\theta}\overline{h(t,x)}\right).
\ee
In particular, for a given steady state of the order parameter written in the polar form 
\be\lbl{ss-order}
h_\infty(x)=R(x) e^{\1 \Phi(x)},
\ee
the velocity field takes the following form
\begin{eqnarray*}
V(t, \theta ,\omega ,x)
 = \omega - KR(x)\sin (\theta -\Phi (x)). 
\end{eqnarray*}
Setting $\partial_\theta \left( V \rho \right)=0$, we find the corresponding steady state
solution of \eqref{MF}
\begin{equation}\lbl{part-sync}
\rho (\theta ,\omega ,x) = \left\{ \begin{array}{ll}
\displaystyle \delta \bigl(\theta -\Phi(x)-\arcsin (\omega /KR(x)) \bigr), & |\omega |\le KR(x), \\
\displaystyle \frac{1}{2\pi} \frac{\sqrt{\omega ^2 - K^2 R(x)^2}}{| \omega - KR(x)\sin (\theta -\Phi (x))|}, 
  & |\omega |>KR(x),
\end{array} \right.
\end{equation}
where $\delta $ stands for the Dirac delta function. 
The stationary solution \eqref{part-sync} has the following interpretation:
the first line describes phase-locked oscillators, while  the second line yields the distribution of 
the drifting oscillators. Thus, solutions of this form may combine phase-locked oscillators and those 
moving irregularly. Such solutions are called partially phase-locked or partially synchronized.
The phase of a phase-locked oscillator at $x$ with a natural frequency $\omega$, 
is given by
\begin{equation}\label{lock}
\theta = \Phi (x) + \arcsin \left( \frac{\omega }{KR(x)}\right),
\end{equation}
provided $|\omega | \le KR(x).$
In Sections~\ref{sec.bif} and \ref{sec.bif2}, we will identify branches of stable equilibria
bifurcating from $K_c^\pm$ in terms of the corresponding values of the order parameter.
Then Equation~\eqref{lock} will be used to describe the corresponding stable phase-locked solutions.

 In Section~\ref{sec.bif}, assuming that $\mu_{max}$ is a simple eigenvalue
of $\mathbf{W},$ we show that the coupled system \eqref{KM} undergoes a supercritical 
pitchfork bifurcation at $K_c^+$. Specifically, we derive an ordinary differential equation
for the order parameter $h$ and show that the trivial solution of this equation looses
stability at $K_c^+$ and gives rise to a stable branch of (nontrivial) equilibria,
corresponding to partially synchronized state (cf.~\eqref{fork}). 
In Section~\ref{sec.bif2}, we consider the onset of synchronization in networks with certain symmetries
(cf.~\eqref{Wsymmetry}). This leads to the bifurcation with a two-dimensional center
manifold. The bifurcation analysis in Sections \ref{sec.bif} and
\ref{sec.bif2} is  illustrated with 
the analysis of the 
KM on Erd\H{o}s-R{\' e}nyi, small-world graphs, and to a class of weighted graphs on a circle.

\section{Preliminaries} \lbl{sec.prelim}
\setcounter{equation}{0}
In the remainder of this paper, we will assume that $K\ge 0$. The case of 
negative $K$ is reduced to that above by switching to $K:=-K$ and $W:=-W$.
Furthermore, without loss of generality we assume that $\mu_{max}>0$.
\subsection{Fourier transform}
We rewrite \eqref{MF} in terms of the complex Fourier coefficients
\be\lbl{Ftransform}
z_j=\int_\SS e^{\1 j\theta} \rho(t,\theta, \omega,x) d\theta, \quad j\in\Z.
\ee
Applying the Fourier transform to \eqref{MF} and using integration by parts, we obtain
\be\lbl{prep-F}
\begin{split}
{\p z_j\over \p t} &= -\int_\SS e^{\1 j\theta} 
{\p \over \p \theta} \left\{ V(t,\theta,\omega,x)
\rho(t,\theta,\omega,x)  \right\} d\theta\\
&=\1 j\int_\SS e^{\1 j\theta} V(t,\theta,\omega,x) 
\rho(t,\theta,\omega,x) d\theta\\
&=\1 j \omega z_j +{jK\over 2} \int_\SS e^{\1 j\theta}  
\left( e^{-\1\theta} h(t,x) - e^{\1 \theta} \overline{h(t,x)} \right) 
\rho(t,\theta,\omega,x) d\theta,
\end{split}
\ee
where
\be\lbl{corder}
\begin{split}
h(t,x) &=\int_I\int_\R\int_\SS W(x,y) e^{\1 \theta} \rho(t,\theta,\omega,y) g(\omega) 
d\theta d\omega dy\\
& =\int_I\int_\R W(x,y) z_1(t,\omega, y) g(\omega) d\omega dy.
\end{split}
\ee

From \eqref{prep-F}, we further obtain
\be\lbl{Fourier}
{\p z_j \over \p t} = \1 j \omega z_j + {jK\over 2}
\left(hz_{j-1} -\overline{h} z_{j+1}\right),\quad j\in \Z.
\ee
By \eqref{normalize-g}, $z_0=1$. Further, $z_{-j}=\bar z_j,$ because $\rho$ is real.
Thus, in \eqref{Fourier} we can restrict to $j\in\N$.

Let 
\be\lbl{def-P}
\begin{split}
\bP f(\omega, x) & = \int_\R \bW[f] (\omega, x)
g(\omega) d\omega\\
&=\int_I\int_\R W(x,y) f(\omega,y) g(\omega) d\omega dy.
\end{split}
\ee
Combining these observations, we rewrite \eqref{Fourier}:
\begin{eqnarray}\lbl{dZ1-1}
{\p\over \p t} z_1 &=& \1\omega z_1 +{K\over 2} \left( \bP z_1
-\overline{\bP z_1} z_2\right),\\
\lbl{dZj-1}
{\p\over \p t} z_j &=& \1 j\omega z_j +{jK\over 2} 
\left( \bP z_1z_{j-1}-\overline{\bP z_1} z_{j+1}\right),
\quad j=2,3,\dots.
\end{eqnarray}

Note that the trivial solution $Z:=(z_1,z_2,\dots)\equiv 0$ is a steady state
solution of \eqref{dZ1-1}, \eqref{dZj-1}. It corresponds to the uniform distribution
$\rho_u=1/(2\pi)$, a constant steady state solution of \eqref{MF}.
Linearizing around $Z\equiv 0,$ we arrive at
\begin{eqnarray}\lbl{dZ1}
{\p\over \p t} z_1 &=& \bT z_1,\\
\lbl{dZj}
{\p\over \p t} z_j &=& \1 j\omega z_j, 
\quad j=2,3,\dots,
\end{eqnarray}
where $\bT$ is a linear operator on $H:=L^2(\R\times I, g(\omega) d\omega dx)$
\be\lbl{def-T}
\bT f= \1\omega f +{K\over 2}\bP f.
\ee

\subsection{The eigenvalue problem}

The multiplication operator 
$\mathbf{M}_{\1\omega}: H\to H$ 
defined by
\be\lbl{def-M}
\mathbf{M}_{\1\omega} f=\1\omega f, \quad \omega \in \R
\ee 
is a closed operator. 
The continuous spectrum of $\mathbf{M}_{\1\omega}$ fills the imaginary axis
\be\lbl{cspectrum}
\sigma_c(\mathbf{M}_{\1\omega})=\1 \mathrm{supp}(g) =\1\R.
\ee
Since $\bP$ is compact (as a Hilbert-Schmidt operator), 
$\bT: H\to H$ is closed 
and  $\sigma_c(\bT)=\1\R$. 

Next we turn to the eigenvalue problem
\be\lbl{EV-T}
\bT f = \lambda f,
\ee
where $\bT$ and $\bP$ are operators on  
$H$ (cf.~\eqref{def-T} and \eqref{def-P}). 
 
We will locate the eigenvalues of $\bT$ through the eigenvalues of $\bW$ (cf.~\eqref{def-W}).
Since $\bW$ is a compact symmetric operator on $L^2(I)$, it has 
a countable set of real eigenvalues with the only accumulation point at zero.
All nonzero eigenvalues have finite multiplicity.

Suppose $\lambda$ is an eigenvalue of $\bT$ and 
$v\in H$ is the corresponding 
eigenfunction. Then a simple calculation yields (cf.~\cite{ChiMed16})
\be\lbl{link}
w={K\over 2} D(\lambda) \bW w,
\ee
where 
\begin{eqnarray}\lbl{def-D}
D(\lambda)&=&\int_\R {g(\omega) d\omega\over \lambda -\1\omega},\\
\lbl{eqn-w}
w &=& \int_\R v(\omega,\cdot) g(\omega)d \omega \in L^2(I).
\end{eqnarray}
Equation \eqref{link} yields the equation for eigenvalues of $\bT$
\be\lbl{character}
D(\lambda)={2\over K\mu},
\ee
where $\mu$ is a nonzero eigenvalue of $\bW$. 

Using \eqref{character}, we establish a one-to-one correspondence 
between the eigenvalues of $\bW$ and those of $\bT$. Specifically,
for every positive eigenvalue of $\bW$, $\mu$, there is a branch of eigenvalues
of $\bT$, 
\be\lbl{lambda-branch}
\lambda=\lambda(\mu,K), \quad K\ge K(\mu):={2\over \pi g(0)\mu},
\ee
such that 
\be\lbl{ends}
\lim_{K\to K(\mu)+0} \lambda(\mu,K) =0+,\quad
\lim_{K\to \infty}\lambda(\mu,K) =\infty.
\ee
Recall that $\mu_{max}$ stands for the largest positive eigenvalue of $\bW$. Then
for $K\in [0,K(\mu_{max}))$ there are no eigenvalues with positive real part.
Furthermore, for small $\varepsilon>0$ and $K\in (K(\mu_{max}), 
K(\mu_{max})+\varepsilon)$ there is a unique positive eigenvalue of $\bT$,
$\lambda(K,\mu_{max}),$ which vanishes as $K\to K(\mu_{max})+0$ (see~\cite{ChiMed16} 
for more details).

\section{The generalized spectral theory}\lbl{sec.spectral}
\setcounter{equation}{0}

The major obstacle in studying stability and bifurcations of the incoherent
state is the continuous spectrum of the linearized problem on the imaginary
axis (cf.~\eqref{cspectrum}). To deal with this difficulty, we develop the generalized spectral theory 
following the treatment of the classical KM in \cite{Chi15a}. Below, we outline
the key steps in the analysis of the generalized eigenvalue problem referring
the interested reader to \cite{Chi15a}, \cite{Chi15b} for missing proofs and further details.

\subsection{The rigged Hilbert space}

In this subsection, we define a rigged Hilbert space (a.k.a. Gelfand triple) \cite{GelVilv4}
\begin{eqnarray*}
X \subset H\subset X^\prime,
\end{eqnarray*}
where $H$ is a Hilbert space, and $X$ is a dense subspace of $H$,
whose topology is stronger than that of $H$.
Throughout this paper, we assume that $X$ is a locally convex
Hausdorff topological  vector space over $\C$ and $X'$ its dual space,
the space  of continuous antilinear functionals on $X$. 
Let $\langle \cdot, \cdot\rangle$ denote the pairing 
between $X^\prime$ and $X,$ i.e.,  for  $l \in X^\prime$ and $f \in X$, 
$\langle l, f \rangle:=l(f)$ stands for the corresponding antilinear functional.
To use the generalized spectral theory  (cf.~\cite{Chi15b}) we also need
$X$ to be a quasi-complete barreled space.

We take $L^2 (\R \times I, g(\omega )d\omega dx )$ as the Hilbert
space $H$ and contruct  $X$ as follows.
Let $\mathrm{Exp}(\beta, n)$ be the set of holomorphic functions on the region 
$\C_n := \{ z\in \C \, | \, \mathrm{Im} (z) \geq -1/n\}$ such that the norm
\begin{equation}
|| \phi ||_{\beta, n} := \sup_{\mathrm{Im} (z) \geq -1/n} e^{-\beta |z|}|\phi (z)|
\label{exp1}
\end{equation}
is finite. With this norm, $\mathrm{Exp}(\beta, n)$ is a Banach space.
The family of spaces $\{ \mathrm{Exp}(\beta, n)\}^\infty_{n=1}$ is an increasing sequence in $n$.
By Montel's theorem, the inclusion $\mathrm{Exp}(\beta, n) \to \mathrm{Exp}(\beta, n+1)$
is a compact operator for any $n\geq 1$.
By Komatsu's theorem \cite{Kom67},  the inductive limit
\begin{eqnarray*}
\mathrm{Exp}(\beta) : = \varinjlim_{n \geq 1} \mathrm{Exp} (\beta, n) 
\left( = \bigcup_{n \geq 1} \mathrm{Exp} (\beta, n)  \right).
\end{eqnarray*}
is a complete Montel space.
In particular, it is a complete barreled (DF) space.
Similarly, the inductive limit $\displaystyle \mathrm{Exp}:= \varinjlim_{\beta \geq 0}\mathrm{Exp}(\beta)$
is a complete barreled (DF) space.
The properties of  $\operatorname{Exp}$ are described in detail in \cite{Chi15a}.

Let $\mathrm{Exp}(\beta, n) \otimes L^2(I)$ be a projective tensor product.
Since the identity map $L^2(I) \to L^2(I)$ is weakly compact, the inclusions
$\mathrm{Exp}(\beta, n) \otimes L^2(I) \to \mathrm{Exp}(\beta, n+1) \otimes L^2(I)$
and $\mathrm{Exp}(\beta) \otimes L^2(I) \to \mathrm{Exp}(\beta +1) \otimes L^2(I)$ are weakly compact operators.
By Komatsu's theorem \cite{Kom67}, the inductive limit $X:= \mathrm{Exp} \otimes L^2(I)$
is a complete barreled (DF) space and $X'$ is a Fr\'{e}chet space.
For every $f \in X$ we have $f(\omega,\cdot)\in L^2(I)$ for each $\omega\in\R$. In addition, 
$f(\omega, x)$ is holomorphic in $\omega$ on the upper half plane, where it can grow at most exponentially.
Then the operator $\bT$ and the rigged Hilbert space $X\subset H \subset X'$ 
satisfy all assumptions of the generalized spectral theory in
\cite{Chi15b}.

Note that if $l \in H$ then
\begin{eqnarray*}
\langle l, f \rangle := (l, f^*)_{L^2(\R\times I)} = 
\int_{I}\!\int_{\R} \! l (\omega,x )f(\omega,x )g(\omega )d\omega dx, 
\end{eqnarray*}
for $f\in X$,
where $f^* (\omega ,x) := \overline{f (\overline{\omega }, x)}$. 
Thus, $l\in H$ can be viewed as an element of $X'$.

\subsection{The generalized eigenvalue problem}

In this subsection we calculate the resolvent
of $\bT$ and spectral projections.
With the rigged Hilbert space defined above, we will view the resolvent as an operator
from $X$ to $X^\prime$.

Below, we will need to construct analytic continuation for certain functions involving
integrals of Cauchy type. For this, we are going to use an implication of the Sokhotski-Plemelj
formulas, which we formulate as a separate statement for convenience.

\begin{lem}\lbl{lem.Sokh} (Sokhotski-Plemelj, cf.~\cite{Gakhov66})
Let $f$ be a complex valued function on $\R$. Suppose $f$ has at most a 
finite number 
of integrable discontinuities. Then
\be\lbl{Cauchy}
F(z)=\int_\R {f(\omega) d\omega \over z- \1\omega}
\ee
is an analytic function in the right and left open half-planes of $\C$. Furthermore,
for $z=x+\1 y$, the following formulas determine the limits of $F(z)$ as $x\to 0\pm$:
\be\lbl{Sokh}
\begin{split}
\lim_{x\to 0\pm}  F(z) &=  \pm\pi f(y) + \1 \operatorname{PV} 
\int_{\1\R} {f(-\1\phi)d\phi\over \phi - \1 y} \\
 &= \pm\pi f(y) -\1\pi H[f](y),
\end{split}
\ee
where $\operatorname{PV}$ stands for the principal value in the sense of Cauchy and
$H[f]$ denotes the Hilbert transform of $f$.
\end{lem}

\begin{cor} \lbl{cor.Sokh} Suppose $f$ is holomorphic on the real axis and admits 
the analytic continuation to the upper half-plane. Then
\be\lbl{extend}
\tilde F(z) =\left\{ \begin{array}{ll}  F(z), & x>0,\\
\lim_{x\to 0+} F(z), & x=0,\\
F(z)+2\pi f(-\1 z), & x<0,
\end{array}
\right.
\ee
is an entire function.
\end{cor}

\subsection{The generalized resolvent}

Our next goal is to compute the resolvent of $\bT$
\be\lbl{def-Res}
\bR(\lambda)=(\lambda -\bT)^{-1}.
\ee
To this end, we first compute $\bR(\lambda)$ for $\Rep(\lambda)>0$ and extend
it analytically to the left half-plane as an operator from $X$ to $X'$. 

In the right half-plane $\Rep(\lambda)>0,$ $\bR(\lambda)$ can be rewritten as follows
\be\lbl{Re-bR}
\bR(\lambda)=A(\lambda)
\left( \bI-{K\over 2} \bP A(\lambda)\right)^{-1} 
= \left( \bI-{K\over 2} A(\lambda) \bP \right)^{-1} A(\lambda ),
\ee
where
\be\lbl{def-bA}
A(\lambda)=\left(\lambda -\1\omega\right)^{-1},
\ee
and $\bI$ stands for the identity operator.
Note that $A(\lambda)$ ceases to exist as the multiplication operator on 
$H$
as $\Rep (\lambda )\to 0$ 
(recall that the imaginary axis is the continuous spectrum of $\mathbf{M}_{\1\omega}$). 
However, it can be extended to the left half-plane as
as an operator $\mathcal{A}: X\to X^\prime$ defined as follows
\be\lbl{def-A}
\langle \mathcal{A}(\lambda) u, v\rangle =\left\{ \begin{array}{ll}
\left(A(\lambda) u, v^\ast\right)_{H}, & \Rep (\lambda )>0,\\
\lim_{\Rep(\lambda )\to 0+} 
\left(A(\lambda) u, v^\ast\right)_{H}, & \Rep (\lambda )=0,\\
\left(A(\lambda) u, v^\ast\right)_{H}+2\pi g(-i\lambda )
\int_{I} u(-\1\lambda,x)v(-\1\lambda,x) dx,  & \Rep (\lambda )<0.
\end{array}
\right.
\ee
 By Corollary~\ref{cor.Sokh}, $\langle \mathcal{A}(\lambda) u, v\rangle$ is an entire
function in $\lambda $ for all $u,v\in X$.  This suggests an appropriate 
generalization of $\bR(\lambda),$ $\mathcal{R}(\lambda): X\to X^\prime$ defined by
\be\lbl{def-R}
\mathcal{R}(\lambda) =\mathcal{A}(\lambda) 
\left( \bI -{K\over 2} \mathbf{P}^\times
\mathcal{A}(\lambda)\right)^{-1}
= \left( \bI -{K\over 2} \mathcal{A}(\lambda) \mathbf{P}^\times
\right)^{-1} \mathcal{A}(\lambda),
\ee
where $\bP^\times: X^\prime\to X^\prime$ is the dual operator of $\mathbf{P}$.

Since $\bT$ has the continuous spectrum on the imaginary axis, $\bR(\lambda)$ 
can not be continued to the left-half plane as an operator on $H$.
We define the generalized eigenvalues of $\bT$ as the singularities of the generalized
resolvent $\mathcal{R}(\lambda).$

\begin{df}\lbl{def.resolvent}
$\lambda\in\C$ is called a generalized eigenvalue of $\bT$ if there is a nonzero
$v\in X^\prime$ such that 
\be\lbl{generalized-EV}
\left(\bI -{K\over 2}\mathcal{A}(\lambda) \mathbf{P}^\times \right)v=0,
\ee
In this case, $v$ is called a generalized eigenfunction.
\end{df}
\begin{rem}\lbl{rem.gEV}
Since the range of the operator $\mathbf{P}^\times : X' \to X'$ is in $X$, 
$\mathcal{A}(\lambda) \mathbf{P}^\times v$ is well-defined for $v\in X'$.
\end{rem}
\begin{rem}\lbl{rem.gEV-2}
The generalized eigenvalues and the corresponding eigenfunctions of 
$\bT$ are, in fact, the eigenvalues and eigenfunctions of the dual of 
$\bT,$ $\bT^\times$ (cf.~\cite{Chi15b}).
\end{rem}

\begin{thm}(cf.~\cite{Chi15b})\lbl{thm_chiba} Let $\lambda\in\C$ be a generalized eigenvalue of 
$\bT$ and $v\in X^\prime$ is the corresponding eigenfunction. Then $\bT^\times v=\lambda v$.
\end{thm}

\begin{rem}
Using \eqref{def-A} and \eqref{generalized-EV}, one can see that the generalized eigenvalues $\lambda =\lambda (\mu, K)$
of $\bT$ are the roots of the following equation
\be\lbl{one-can-see}
{2\over K\mu}=\mathcal{D}(\lambda),
\ee
where $\mu$ is a nonzero eigenvalue of $\bW$ and
\be\lbl{def-D2}
\mathcal{D}(\lambda)= \left\{ \begin{array}{ll}
D(\lambda ), & \Rep(\lambda )>0,  \\
\displaystyle\lim_{\Rep (\lambda ) \to 0+}D(\lambda ), & \Rep(\lambda )=0,  \\
D(\lambda )+2\pi g(-i\lambda ), & \Rep(\lambda )<0.
\end{array} \right.
\ee
The right hand side of \eqref{def-D2} is an entire function (cf. Corollary~\ref{cor.Sokh}).
For $\Rep(\lambda ) > 0$, \eqref{one-can-see} is reduced to 
the equation for the eigenvalues of $\bT$ (cf.~\eqref{character}).
In this case, the corresponding generalized eigenfunction  
$v$ is included in $L^2(\R \times I, gd\omega dx),$ i.e., $\lambda$ is an eigenvalue of $\bT$.
On the other hand, for $\Rep (\lambda ) \leq 0$, 
the generalized eigenfunction $v$ is not in $H$
but is an element of the dual space $X^\prime$.
\end{rem}

Since the generalized eigenvalue of $\bT,$ $\lambda,$ is a  root of
\eqref{one-can-see}, $\mathcal{R}(\lambda )u$ is an $X'$-valued meromorphic function for
each $u\in X$. For $\Rep~\lambda >0$, it coincides with the
restriction of $\bR(\lambda)$ to $X$. Thus, $\mathcal{R}(\lambda )$ is
a meromorphic continuation of $\bR(\lambda)$ from the right half-plane to the left-half plane as an $X'$-valued
operator.

\subsection{The generalized Riesz projection}

Let $\mu$ be a positive eigenvalue of $\bW$ and $w \in L^2(I)$ be 
the corresponding eigenfunction. The largest positive eigenvalue 
of $\bW$ and the corresponding eigenfunction are denoted by $\mu_{max}$ and
$w_{max}$ respectively. 
For every 
$K>K_c^+=2/(\pi g(0)\mu_{max})$ there is a real positive eigenvalue of $\bT,$ 
$\lambda=\lambda(\mu, K)$. The corresponding eigenfunction is given by
\be\lbl{eigenvectorW}
v(\omega, x)={K\over 2}{w(x)\over \lambda -\1\omega}.
\ee

As $K$ approaches the critical value $K_c^+$ from above, 
the eigenvalue $\lambda(\mu_{max}, K)$ converges to $0+$
along the real axis and at $K=K_c^+$ it hits the continuous spectrum on the
imaginary axis. The corresponding eigenfunction approaches the critical
vector
\be\lbl{lim-eigen}
X^\prime \ni v_c^+ := {K^+_c\over 2}\wlim {w_{max}\over \lambda -\1\omega},
\ee
where $\operatorname{lim*}$ stands for the limit in $X^\prime$ with respect
to the weak dual topology \footnote{ 
$\{ l_n \} \subset X^\prime$ converges to $l \in X'$
if $\langle l_n, f \rangle \in \C$ tends to $\langle l, f \rangle$ for every $f \in X$.}
, i.e., 
the action of $v_c^+\in X^\prime$ on $u\in X$ is given by
\be\lbl{critical-action}
\begin{split}
\langle v_c^+, u\rangle &= {K_c^+\over 2} \l\int_{\R\times I} 
{w_{max}(x) g(\omega)\over \lambda -\1\omega} u(\omega,x) d\omega dx\\
&= {K_c^+\over 2} \l \int_\R {\left(w_{max}, u^*(\omega,\cdot)\right)_{L^2(I)} g(\omega)d\omega   
\over \lambda -\1\omega} .
\end{split}
\ee
Let $\lambda\in \C$ be a generalized eigenvalue of $\bT$. Then the generalized
Riesz projection $\bPi_\lambda : X\to X'$ is defined by
\be\lbl{Riesz}
\bPi_\lambda ={1\over 2\pi \1} \int_{\gamma(\lambda)} \mathcal{R}(z) dz,
\ee
where $\gamma(\lambda)$ is a simple closed curve around $\lambda$ 
oriented counterclockwise that does not 
encircle or intersect the rest of the spectrum. Below, we shall refer to such curves
as contours.
The image of $\bPi_\lambda $ gives the generalized eigenspace of $\lambda $~\cite{Chi15a}.

\begin{thm}\lbl{thm.Riesz}
Suppose  the algebraic and geometric multiplicities of $\mu_{max}$ coincide.
Then the generalized Riesz projection of $\lambda=0,$ the generalized eigenvalue of 
$\bT$ for $K=K_c^+$, has the following form
\be\lbl{critical-Riesz}
\begin{split}
\bPi_0 & = -\wlim \left(  D^\prime (\lambda)^{-1}
A(\lambda)
\tilde\bPi_{\mu_{max}} \bD(\lambda)
\right)\\
&= g_1 \wlim \left(
A(\lambda)
\tilde\bPi_{\mu_{max}} \bD (\lambda) \right),
\end{split}
\ee
where $g_1= -\lim_{\lambda\to 0+} D^\prime (\lambda)^{-1}$ is a positive constant,
$A(\lambda)$ was defined in \eqref{def-bA}, and $\tilde\bPi_{\mu}$ stands for 
the Riesz projection onto the eigenspace of $\bW$ corresponding to the eigenvalue $\mu$. 
The operator $\bD(\lambda)$ on $H$ is defined by
\be\lbl{def-bD}
\bD(\lambda) v= 
\int_\R {v(\omega,\cdot) g(\omega) d\omega\over \lambda -\1\omega},\quad
v\in L^2(\R\times I, gd\omega dx).
\ee
\end{thm}

The proof of Theorem~\ref{thm.Riesz} relies on three technical lemmas. Below we 
state and prove these lemmas first and then prove the theorem. 

\begin{lem} \lbl{lem.Resolvent}
Let $\Rep (\lambda )>0$ then
\be\lbl{rewrite-Res}
\mathbf{R}(\lambda) v=  A(\lambda) v +{K\over 2} 
A(\lambda) \bW \left(\bI -{K\over 2} D(\lambda) \bW \right)^{-1}
\bD(\lambda) v, \quad v \in H.
\ee
\end{lem}
\begin{proof} By definition of $\bR(\lambda)$ \eqref{def-Res}, for any 
$v\in H$, we have
$$
\left(\lambda -\1\omega -{K\over 2}\bP\right) \bR(\lambda) v= v,
$$
and, thus,
\be\lbl{R-1}
\bR(\lambda) v= A(\lambda) v +{K\over 2}  
A(\lambda) \bP\bR(\lambda) v.
\ee

Using Fubini theorem, from \eqref{def-P} we have
\be\lbl{using-Fubini}
\begin{split}
\bP\bR(\lambda) v & = \bW\left[ \int_\R (\bR(\lambda)v)(\omega,\cdot) g(\omega) d\omega\right]\\
&=: \bW \mathbf{Q} v.
\end{split}
\ee
On the other hand, integrating both sides of \eqref{R-1} against $g(\omega)d\omega$, we obtain
\be\lbl{R-2}
\begin{split}
\mathbf{Q} v &=
\int_\R (\bR(\lambda) v)(\omega, \cdot) g(\omega) d\omega\\
&= \bD(\lambda) v +{K\over 2}  D(\lambda) \int_{\R\times I}
W(\cdot, y) \left( \bR(\lambda) v\right) (\omega, y) g(\omega)d\omega dy\\
&= \bD(\lambda) v +{K\over 2}  D(\lambda) \bW \mathbf{Q} v.
\end{split}
\ee
and
\be\lbl{R-3}
\mathbf{Q} =\left(\bI - {K\over 2} D(\lambda) \bW\right)^{-1} \bD(\lambda).
\ee
Plugging \eqref{R-3} into  \eqref{using-Fubini}, we have
\be\lbl{R-3-and-Fubini}
\bP\bR(\lambda) v=\bW \left(\bI - {K\over 2} D(\lambda) \bW\right)^{-1} \bD(\lambda) v.
\ee
The combination of \eqref{R-1} and \eqref{R-3-and-Fubini} yields \eqref{rewrite-Res}.
\end{proof}

\begin{lem}\lbl{lem. Riesz2}
Let $\lambda=\lambda(\mu, K)>0$ be an eigenvalue of $\bT$ corresponding to the 
positive eigenvalue of $\bW$, $\mu$, and $K>K_c^+$, and suppose that  
the geometric and algebraic multiplicities of $\mu$ coincide.

Then 
\be\lbl{Riesz2+}
{\mathbf \Pi}_\lambda = - D^\prime (\lambda)^{-1} A(\lambda)
\tilde{\mathbf \Pi}_\mu \bD(\lambda),
\ee
provided $D'(\lambda ) \neq 0$,
where $\tilde{\mathbf \Pi}_\mu$ is the Riesz projection onto the eigenspace of $\bW$ 
corresponding to $\mu$.
\end{lem} 

\def\bI {\operatorname{\mathbf I}}

\begin{proof} As before, let $\gamma(\lambda)$ denote a contour around 
$\lambda$. From \eqref{rewrite-Res}, we have
\be\lbl{int-Res}
\int_{\gamma(\lambda)} \bR(z)dz = {K\over 2} 
\int_{\gamma(\lambda)} A(z) \bW\left( \bI -{K\over 2} 
D(z)\bW\right)^{-1} \bD(z)dz.
\ee
We change variable in the integral on the right--hand side to $\zeta=2(K D(z))^{-1}$.
By deforming the contour $\gamma(\lambda)$ if necessary, we can always achieve
$D^\prime(z)\neq 0$ for $z\in \gamma(\lambda),$ so that this change of variable 
$\zeta=\zeta(z)$ is well defined. Under this transformation, $\gamma(\lambda)$ 
is mapped to $\tilde\gamma(\mu)$, a contour around $\mu$. Thus, we have
\be\lbl{int-Res}
\int_{\gamma(\lambda)} \bR(z)dz = -\int_{\tilde\gamma(\mu)}
A(z(\zeta)) \bW \left(\zeta -\bW\right)^{-1} \bD(z(\zeta))
{d\zeta\over \zeta D^\prime (z(\zeta))}.
\ee
Since the algebraic and geometric multiplicities of $\mu$ are
equal, the singularity of $\left(\zeta -\bW\right)^{-1}$ at $\zeta=\mu$ is a simple 
pole, and the other factor in the integrand of the above is regular at $\zeta = \mu$.
Therefore, the right--hand side of \eqref{int-Res} simplifies to
\be\lbl{int-Res-1}
\int_{\gamma(\lambda)} \bR(z)dz =   -\frac{1}{\mu D^\prime (\lambda)} A(\lambda)
\bW  \left(\int_{\tilde\gamma(\mu)} 
\left(\zeta -\bW\right)^{-1} \bD (z(\zeta))d\zeta\right).
\ee
By multiplying both sides of \eqref{int-Res-1} by $(2\pi i)^{-1}$, we have
$$
\operatorname{\mathbf \Pi}_\lambda =   -\frac{1}{\mu D^\prime (\lambda)} A(\lambda)
\bW \tilde\bPi_\mu \bD(\lambda).
$$
Finally, since $\tilde{\operatorname{\mathbf \Pi}}_\mu$ is the projection 
on the eigensubspace of $\bW,$
$$
\bW \tilde{\operatorname{\mathbf \Pi}}_\mu \bD(\lambda)=
\mu \tilde{\operatorname{\mathbf \Pi}}_\mu \bD(\lambda).
$$
Thus,
\be\lbl{Bpi-lambda}
\operatorname{\mathbf \Pi}_\lambda =  
- (D^\prime (\lambda))^{-1} A(\lambda)
 \tilde{\operatorname{\mathbf \Pi}}_\mu \bD(\lambda).
\ee
\end{proof}

\begin{lem}\lbl{lem.byparts}
\be\lbl{byparts}
\lim_{z\to 0+} \int_\R {g(\omega) d\omega \over (z-\1\omega)^{n+1}} =
{\1^n\pi \over n!} \left( g^{(n)}(0) -\1 H[g^{(n)}](0)\right).
\ee
\end{lem}
\begin{proof}
Using integration by parts $n$ times, we obtain
$$ 
\int_\R {g(\omega) d\omega \over (z-\1\omega)^{n+1}} =
{\1^n \over n!} \int_\R {g^{(n)}(\omega) d\omega \over z-\1\omega}.
$$
The application of Lemma~\ref{lem.Sokh} to the integral on the 
right-hand side yields \eqref{byparts}.
\end{proof}

Below will need the following implications of Lemma~\ref{lem.byparts}.
\begin{cor}\lbl{cor.byparts}
\begin{eqnarray} \lbl{Dprime}
\lim_{z \to 0+} D^\prime(z) &=& -\pi H[g^\prime](0)<0,\\
\lbl{third}
\lim_{z\to 0+} \int_\R {g(\omega) d\omega \over (z-\1\omega)^3} 
&=& {-\pi\over 2} g^{\prime\prime}(0).
\end{eqnarray}
\end{cor}
\begin{proof}
Differentiating $D(z)$  and using \eqref{byparts}, for 
$z$ off the imaginary axis we have
\be\lbl{by-parts}
D^\prime (z)=  -\int_\R {g(\omega)d\omega \over (z-\1\omega)^2}=
-\1 \int_\R {g^\prime (\omega)d\omega \over  z-\1\omega}.
\ee
The integral on the right--hand side is of Cauchy type and Lemma~\ref{lem.Sokh}
applies. By \eqref{Sokh},
\be\lbl{apply-Sokh}
\lim_{z\to 0+} D^\prime (z) = -i\pi g^\prime (0) -\pi H[g^\prime] (0).
\ee
Since $g$ is even, $g^\prime(0)=0$ and $g^\prime$ is odd. Because $g$ 
is also nonnegative and unimodal $g^\prime(x) \le 0, \; x>0.$ Thus,
\be\lbl{g-even}
\begin{split}
H[g^\prime](0) & = {-1\over \pi}\operatorname{PV}
\int_{-\infty}^\infty  {g^\prime(s)ds\over s} \\
&=
{-2\over \pi}\lim_{\epsilon\to 0+} \int_\epsilon^\infty {g^\prime(s)ds\over s} >0.
\end{split}
\ee
The combination of \eqref{Sokh} and \eqref{g-even} yields \eqref{Dprime}.

Likewise, \eqref{third} follows from Lemma~\ref{lem.byparts} 
for $n=2$ and Lemma~\ref{lem.Sokh}.
\end{proof}

\begin{proof}[Proof of Theorem~\ref{thm.Riesz}]
Theorem~\ref{thm.Riesz} follows from \eqref{Bpi-lambda} and \eqref{Dprime}.
\end{proof}

\section{Asymptotic stability of the incoherent state}\lbl{sec.asympt}
\setcounter{equation}{0}

We now return to the problem of stability of the incoherent state. Recall that in the Fourier
space the incoherent state corresponds to the trivial solution $Z=(z_1, z_2,\cdots )=0$. The linearization 
about $Z=0$ shows that it is a neutrally stable equilibrium of \eqref{dZ1}, \eqref{dZj}
for $0\le K<K_c^+$. There are no eigenvalues of $\bT$ for these values of $K$
and the continuous spectrum fills out the imaginary axis.
Nonetheless, we show that the incoherent state is asymptotically stable with respect 
to the weak dual topology.

\begin{thm}\lbl{thm.asympt}
For $K\in [0, K_c^+)$ the trivial solution of \eqref{dZ1}, \eqref{dZj} is an asymptotically 
stable equilibrium for initial data from $X\subset X^\prime$ with
respect to the weak dual topology on $X^\prime$.
\end{thm}
\begin{rem}\lbl{rem.asympt}
The stability with respect to the weak dual topology is weaker than that with 
respect to the topology of the Hilbert space $H$. 
Still it is a natural topology  for the problem at hand.
In particular, Theorem~\ref{thm.asympt} implies that the order parameter evaluated 
on the trajectories of the linearized problem tends to $0$ as $t\to\infty$.
\end{rem}
\begin{figure}
\begin{center}
\includegraphics[height=2.0in,width=2.5in]{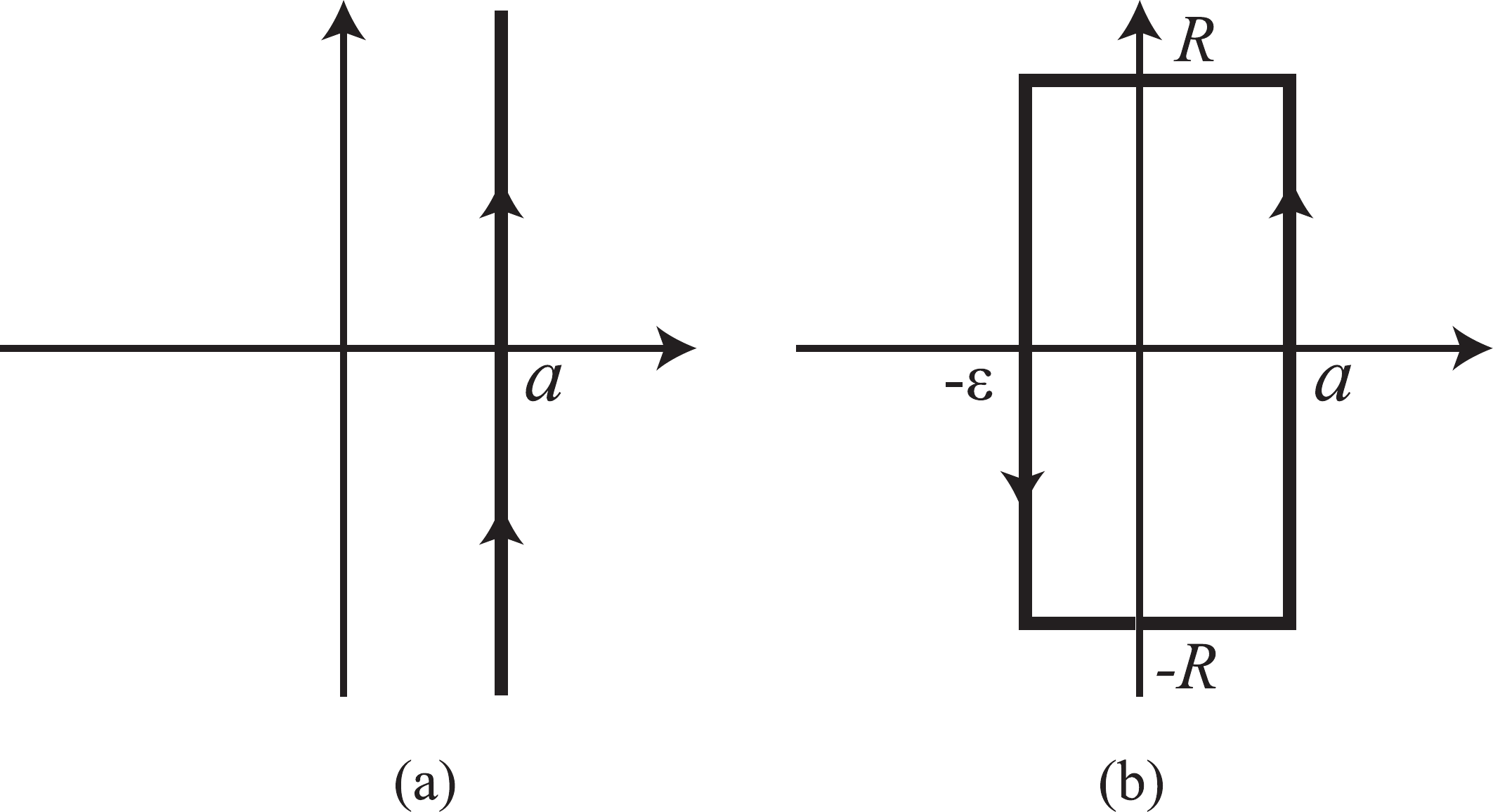}
\caption{Deformation of the integral path for the Laplace inversion formula.}
\label{f.1}
\end{center}
\end{figure}
\begin{proof}
Integrating \eqref{dZj} subject to $z_j(0,\cdot)\in X$, we have
$$
z_j(t,\cdot)=e^{\1 j\omega t} z_j(0,\cdot), \; j\ge 2.
$$
By the Riemann-Lebesgue lemma,
\be\lbl{j-ge-0}
\langle z_j(t,\cdot), \psi(\cdot)\rangle =\int_I \int_\R e^{\1 j\omega t} z_j(0,\omega, x)
\psi(\omega,x) d\omega dx \to 0, \;\mbox{as}\; t\to\infty, \quad \forall \psi\in X.
\ee
We now turn to \eqref{dZ1}.
By the Hille-Yosida theory, operator $\bT$ generates a $C^0$-semigroup $e^{\bT t}$, 
which can be computed using inverse Laplace transform (cf.~\cite{Chi15a}):
\be\lbl{ebT}
e^{t\bT} = \lim_{b\to \infty} \frac{1}{2\pi i} \int^{a+ib}_{a-ib}\! 
e^{\lambda t}(\lambda -\bT)^{-1}d\lambda,  \quad t>0,
\ee
where $a>0$ is arbitrary. Thus, the (continuous) spectrum of $\mathbf{T}$ lies to the 
left of the integration path along $x=a$ (see Fig.~\ref{f.1}\textbf{a}).

For arbitrary $\phi,\psi \in H$, we have
\be\lbl{pairing}
\left( e^{t\bT} \phi, \psi\right)_H=  \lim_{b\to \infty} \frac{1}{2\pi i} 
\int^{a+ib}_{a-ib}\! e^{\lambda t}\left( (\lambda -\bT)^{-1} \phi, 
\psi \right)_H  d\lambda.
\ee
For $\phi, \psi \in X$, $\left( (\lambda -\bT)^{-1} \phi, \psi \right)_H$ is an analytic function
in the right half--plane, which can be extended to the entire complex plane as a 
meromorphic function $\langle \mathcal{R}(\lambda )\phi, \psi \rangle$.
Thus, 
\be\lbl{pairing-1}
\langle e^{t\bT} \phi, \psi \rangle = 
 \lim_{b\to \infty} \frac{1}{2\pi i} 
\int^{a+ib}_{a-ib}\! e^{\lambda t}  \langle \mathcal{R}(\lambda )\phi, \psi \rangle  d\lambda
\quad \forall \phi, \psi \in X.
\ee 

Let $K\in [0, K_c^+)$ be fixed. Next we claim that one can choose $\varepsilon=\varepsilon (K)>0$
such that there are no generalized eigenvalues of $\bT$ on or inside the contour
$$
C_{\varepsilon,R}:\quad a-\1R\to a+\1 R \to -\varepsilon +\1 R\to -\varepsilon -\1 R\to a-\1R
\quad \mbox{(Fig.~\ref{f.1}\textbf{b})}
$$
for every $R>0$. To construct $C_{\varepsilon,R}$ with the desired property,
we first fix $\delta>0$. Then we recall that 
generalized eigenvalues of $\bT$ satisfy \eqref{one-can-see}. From \eqref{def-D2}, under
our assumptions on $g$, there exists $R_0=R_0(\delta)$ such that there are no roots of 
\eqref{one-can-see} in the region 
$$
D^+_{R_0,\delta}=\{z\in \C:\; |z|\ge R_0\; \& -\delta < \mathrm{Re}(z) \leq a\},
$$
because (\ref{one-can-see}) can be reduced to $2/(K\mu)=O(1/ |\lambda|)$ 
in $D^+_{R,\delta}$ for $R\gg 1$.
On the other hand, $\mathcal{D}(\lambda)$ is holomorphic. Thus, the 
set of roots of \eqref{one-can-see} (i.e., the set of generalized eigenvalues) does 
not have accumulation points in
$$
D^-_{R_0,\delta}=\{z\in \C:\; |z|\leq R_0\; \& -\delta < \mathrm{Re}(z) \leq a \}.
$$
Thus, we can choose $\varepsilon>0$ such that
there are no generalized eigenvalues in $D^+_{R_0, \varepsilon } \cup D^-_{R_0, \varepsilon }$. 
This completes the construction of $C_{\varepsilon,R}$ with the desired property for any $R>0$.

By the Cauchy Integral theorem,
\be\lbl{CIT}
\oint_{C_{\varepsilon,R}} e^{\lambda t}  \langle \mathcal{R}(\lambda )\phi, 
\psi \rangle  d\lambda=0 \quad \forall \phi, \psi \in X,
\ee 
for any $R>0$,
and
\be\lbl{useCIT}
\int_{a-\1 R}^{a+\1 R} e^{\lambda t} \langle \mathcal{R}(\lambda )\phi, 
\psi \rangle d\lambda =
\left( \int_{-\varepsilon-\1 R}^{-\varepsilon+\1 R} - \int^{-\varepsilon+\1 R}_{a+\1 R}-
\int_{-\varepsilon -\1 R}^{a-\1 R}\right) e^{\lambda t}
\langle \mathcal{R}(\lambda )\phi, 
\psi \rangle d\lambda.
\ee
The integral on the left--hand side of \eqref{useCIT} exists, by the Hille-Yosida theory.
Therefore, the integrals on the right--hand exist too.
Below, we show that the last two integrals on the right--hand side of \eqref{useCIT}
tend to $0$ as $R\to\infty$.
Sending $R\to\infty$ in \eqref{useCIT} and using \eqref{pairing-1}, we 
arrive at
\be
\begin{split}
\langle e^{t\bT} \phi, \psi \rangle &= 
 \lim_{R\to \infty} \frac{1}{2\pi i} 
\int^{-\varepsilon+iR}_{-\varepsilon-iR}\! e^{\lambda t}  
\langle \mathcal{R}(\lambda )\phi, \psi \rangle  d\lambda \\
&=  \frac{e^{-\varepsilon t}}{2\pi i} \lim_{R\to \infty} \int^{R}_{-R}i e^{i \lambda t} 
\langle \mathcal{R}(i\lambda -\varepsilon  )\phi, \psi \rangle d\lambda \\
&= o(e^{-\varepsilon t}), \quad \forall \phi, \psi \in X
\end{split}
\ee
as $t\to \infty$ because the integral
$$
\lim_{R\to \infty}  \int^{R}_{-R}ie^{i\lambda t} 
\langle \mathcal{R}(i\lambda -\varepsilon  )\phi, \psi \rangle d\lambda
$$
exists and also tends to zero as $t\to \infty$ due to the Riemann-Lebesgue lemma.

\end{proof}

It remains to prove the following lemma.
\begin{lem}\lbl{lem.int-to-zero}
For $K\in [0, K^+_c)$, 
\be\lbl{integrals}
\lim_{R\to\infty}  \int^{-\varepsilon+\1 R}_{a+\1 R} e^{\lambda t}
\langle \mathcal{R}(\lambda )\phi, \psi \rangle  d\lambda = 
\lim_{R\to\infty}  \int_{-\varepsilon -\1 R}^{a-\1 R} e^{\lambda t}
\langle \mathcal{R}(\lambda )\phi, \psi \rangle  d\lambda =0 \quad \forall \phi, \psi \in X.
\ee
\end{lem}
\begin{proof}
We show that the integral $\int^{-\varepsilon +iR}_{a+iR}\! e^{\lambda t} \langle \mathcal{R}(\lambda )\phi, \psi \rangle d\lambda$
tends to zero as $R\to \infty$.
The second integral $\int^{a-iR}_{-\varepsilon -iR} e^{\lambda t}\langle \mathcal{R}(\lambda )\phi, \psi \rangle d\lambda$
can be treated in the same way.
Further, we decompose the integral into two integrals as
\begin{eqnarray}
\int^{-\varepsilon +iR}_{a+iR}\! e^{\lambda t} \langle \mathcal{R}(\lambda )\phi, \psi \rangle d\lambda
 = \int^{iR}_{a+iR}\! e^{\lambda t} \langle \mathcal{R}(\lambda )\phi, \psi \rangle d\lambda
  + \int^{-\varepsilon +iR}_{iR}\! e^{\lambda t} \langle \mathcal{R}(\lambda )\phi, \psi \rangle d\lambda.
\label{lem4.5}
\end{eqnarray}
We show that the first integral on the right hand side tends to zero as $R\to \infty$.
For $\mathrm{Re}(\lambda ) > 0$,  we have
\begin{eqnarray*}
\langle  \mathcal{R}(\lambda) \phi, \psi \rangle  =  \langle A(\lambda) \phi,\psi \rangle  +{K\over 2} 
\langle  A(\lambda) \bW \left(\bI -{K\over 2} D(\lambda) \bW \right)^{-1}
\bD(\lambda) \phi, \psi \rangle,
\end{eqnarray*}
see (\ref{rewrite-Res}).
For the first term, we have
\begin{eqnarray*}
& & \int^{iR}_{a+iR}\! e^{\lambda t} \langle A(\lambda )\phi, \psi \rangle d\lambda \\
&=& e^{iRt} \int^{0}_a e^{\lambda t} \int_I \int_{\R} \frac{1}{\lambda +i(R-\omega )} \phi (\omega ,x)\psi (\omega ,x)g(\omega )d\omega dxd\lambda .
\end{eqnarray*}
Since the integral above is finite, for any $\varepsilon _0 > 0$, there exists $L>0$ such that
\begin{eqnarray*}
\left| \int^{0}_a e^{\lambda t} \int_I \int_{|\omega |>L} \frac{1}{\lambda +i(R-\omega )} \phi (\omega ,x)\psi (\omega ,x)g(\omega )d\omega dxd\lambda  \right| < \varepsilon _0
\end{eqnarray*}
On the other hand, the integrand 
\begin{eqnarray*}
e^{\lambda t} \frac{1}{\lambda +i(R-\omega )} \phi (\omega ,x)\psi (\omega ,x)g(\omega ) \to 0,
\end{eqnarray*}
as $R\to \infty$ uniformly in $x\in I,\, \omega \in (-L,L)$ and $\lambda \in (0,a)$.
This implies that the integral 
$$
\int^{iR}_{a+iR}\! e^{\lambda t} \langle A(\lambda )\phi, \psi \rangle d\lambda \to 0, \quad
\mbox{as}\;R\to \infty.
$$

Consider
\begin{eqnarray*}
\widetilde{\phi}_\lambda  = \bW \left(\bI -{K\over 2} D(\lambda) \bW \right)^{-1}\bD(\lambda) \phi.
\end{eqnarray*}
The singularity of $\widetilde{\phi}_\lambda$ is a generalized
eigenvalue of $\bT$ (cf.~\eqref{link}).
For $0<K<K_c^+$, there are no generalized eigenvalues  of $\bT$ in the right half-plane and 
on the imaginary axis.
Further, $D(\lambda ) \to 0$ as $|\lambda | \to \infty$ and, thus,
$\bD(\lambda)\to 0$ too. 
This shows that $\widetilde{\phi}_\lambda$ is bounded uniformly in
$\lambda $ on the region 
$\mathrm{Re}(\lambda ) \geq 0$.
By replacing $\phi$ with $\widetilde{\phi}_\lambda$ in the first estimate of the integral of 
$\langle A(\lambda) \phi,\psi \rangle$, we find that
$\int^{iR}_{a+iR}\! e^{\lambda t} \langle A(\lambda )\widetilde{\phi}_\lambda , \psi \rangle d\lambda$
tends to zero.
This shows that $\int^{iR}_{a+iR}\! e^{\lambda t} \langle \mathcal{R}(\lambda )\phi, \psi \rangle d\lambda$
decays to zero as $R\to \infty$.
The second integral in (\ref{lem4.5}) is analyzed in similarly. This completes the proof of 
Lemma~\ref{lem.int-to-zero}. 
\end{proof}

\section{Bifurcation with a one-dimensional null space}\lbl{sec.bif}
\setcounter{equation}{0}

In the previous section, we proved asymptotic stability of the equilibrium at the origin 
of the linearized system \eqref{dZ1}, \eqref{dZj} for $K\in [0, K_c^+).$ 
On the other hand,
for $K>K_c^+$ there is a positive eigenvalue in spectrum of the linearized problem 
(cf.~\cite{ChiMed16}). This signals a bifurcation at $K_c^+$. This bifurcation is analyzed
in this present section. As in the classical KM, the loss of stability of the incoherent state
at $K_c^+$ and the development of partial synchronization for $K>K_c^+$ is best seen
in terms of the order parameter.

Throughout this section, we assume that the largest positive
eigenvalue $\mu_{max}$ of 
$\mathbf{W}$
with the eigenfunction $w_{max}$ is simple. Furthermore, we assume that at $K_c^+$
there is a (one-dimensional) smooth center manifold of the equilibrium 
at the origin of \eqref{dZ1-1},  \eqref{dZj-1}\footnote{The proof of existence of the 
center manifold is a technical problem and is beyond the scope of this paper
(see \cite{Chi15a} for the proof of existence of the center manifold in the 
original KM).}. Under these assumptions, below we show that the order parameter
undergoes a supercritical pitchfork bifurcation at $K_c^+$. The stable branch of equilibria 
bifurcating from $0$ is given by
\be\lbl{fork}
h_\infty (K)=
{ g(0)^2 \pi^{3/2} \over \sqrt{-g^{''}(0) }} \mu_{max}^{3/2} 
\sqrt{\frac{1}{C(x)}}\sqrt{K-K_c^+} +o(\sqrt{K-K_c^+}), \quad K>K_c^+,
\ee 
where
\be\lbl{C-of-x}
C(x):=\frac{ \tilde \bPi_{\mu_{max}}  (|w_{max}|^2 w_{max}) }{|w_{max}|^2 w_{max}}.
\ee
Formula \eqref{fork} generalizes the classical Kuramoto's formula describing
the pitchfork bifurcation in the all-to-all coupled model to the KM on graphs.
The network structure enters into the description of the pitchfork bifurcation
through the largest eigenvalue $\mu_{max}$ and the corresponding eigenspace.

\subsection{Preparation}
Throughout this section, we assume that $\mu_{max}$ is a simple
eigenvalue of $\mathbf{W}$.
Let $K=K_c^++\epsilon$ with $0<\epsilon\ll 1$ and rewrite \eqref{dZ1-1},\eqref{dZj-1}
as follows
\begin{eqnarray}\lbl{dZ1-2}
{\p\over \p t} z_1 &=& \bT_0 z_1 +{\epsilon\over 2} \bP z_1
-{K\over 2} \overline{\bP z_1} z_2,\\
\lbl{dZj-2}
{\p\over \p t} z_j &=& \1 j\omega z_j +{jK\over 2} 
\left( \bP z_1z_{j-1}-\overline{\bP z_1} z_{j+1}\right),
\quad j=2,3,\dots,
\end{eqnarray}
where $\bT_0$ is $\bT$ evaluated at $K=K_c$ and $\bT = \bT_0 + \epsilon \bP/2$.

For small $\epsilon>0$,
the equilibrium of \eqref{dZ1-2}, \eqref{dZj-2} at the origin 
has a $1D$ unstable manifold. 
We reduce the dynamics on the $1D$ unstable manifold, which we approximate 
by the center manifold of the origin for $K=K_c^+$, i.e., for $\epsilon=0$.
For the latter, we assume $z_k=h_k(z_1), k=2,3,\dots,$ on the center manifold, where $h_k$ 
are smooth functions such that $h_k(0)=h_k^\prime(0)=0$.

Let ${\mathbf \Pi}_0$ be the projection to the eigenspace of $\lambda =0$ spanned by $v_c^+$ (cf.~Section 3.4).
To track the evolution on the slow manifold we adopt the following Ansatz:
\begin{eqnarray}\lbl{AN}
z_1 &=& {\mathbf \Pi}_0 z_1 + (\bI -{\mathbf \Pi}_0 )z_1 
 =\alpha c(t) v_c^+ +O(\alpha^2),\\
\lbl{stable}
z_k &= &h_k(z_1)=O(\alpha^2),\quad k=2,3,\dots,\\
\lbl{eps-a2}
\epsilon &=& \alpha^2,
\end{eqnarray}
where $\alpha>0$ is a small parameter, $c(t)$ is the coordinate along
the center manifold, and $v_c^+$ is the generalized eigenfunction of $\bT_0$
corresponding to the zero eigenvalue (cf.~\eqref{lim-eigen}). The Ansatz
\eqref{AN}-\eqref{eps-a2} follows right away once existence of the center manifold
is shown.

We will start by deriving several auxiliary facts that follow from
the Ansatz \eqref{AN}-\eqref{eps-a2}. First, using \eqref{AN}-\eqref{eps-a2} and Theorem~\ref{thm_chiba},
from \eqref{dZ1-2}, we have
\be\lbl{estimate-dz1}
\dot z_1 = \bT_0z_1 +O(\alpha^2)
 = \bT_0^\times (\alpha c(t)v_c^+) + O(\alpha ^2)=O(\alpha^2).
\ee

Next, we estimate the order parameter.
\begin{lem}\lbl{lem.order}
\be\lbl{order}
h(t,x) = \alpha c(t) w_{max}(x) +O(\alpha^2).
\ee
\end{lem}
\begin{proof}
\be\lbl{order-1}
\begin{split}
h &=\mathbf{P}z_1 = \mathbf{P} \left(\alpha c(t) v_c^+ +O(\alpha^2)\right)\\
&= \alpha c(t) {K_c^+\over 2}  \l \int_\R\int_I {W(x,y) w_{max}(y) g(\omega)
\over \lambda -\1\omega} dyd\omega +O(\alpha^2).
\end{split}
\ee

Applying the Fubini theorem, \eqref{one-can-see} and \eqref{def-D}, we have
\be\lbl{order-2}
\begin{split}
h &= \alpha c(t) {K_c^+\over 2}  (\bW w_{max}) D(0+) + O(\alpha^2)\\
& = \alpha c(t) {K_c^+\mu_{max}\over 2} D(0+) w_{max} + O(\alpha^2)\\
&=\alpha c(t) w_{max} + O(\alpha^2).
\end{split}
\ee
\end{proof}

\begin{lem}\lbl{lem.z2}
\be\lbl{z2}
z_2= \left({\alpha c(t) K_c^+\over 2}\right)^2 
\wlim {w_{max}^2\over (\lambda-\1\omega)^2}+ O(\alpha^3).
\ee
\end{lem}
\begin{proof}
Using \eqref{AN}-\eqref{eps-a2} and \eqref{estimate-dz1}, we obtain
\be\lbl{hot}
\begin{split}
\dot z_2 & = h_2^\prime (z_1) \dot z_1 = O(\alpha^3),\\
(\overline{\bP z_1})z_3&= O(\alpha^3).
\end{split}
\ee
By plugging  \eqref{hot} into \eqref{dZj-2} for $j=2$, we obtain
\be\lbl{we-obtain}
0=2\1\omega z_2 + K(\bP z_1)z_1 +O(\alpha^3).
\ee
Next we plug in the expressions for $z_1, \bP z_1$, and  $z_2$ 
(see \eqref{AN}, \eqref{order}, \eqref{z2})
into  \eqref{we-obtain} to verify that they satisfy this equation up to  $O(\alpha^3)$ 
terms. Specifically, we have
\begin{equation*}
\begin{split}
 2\1\omega z_2 + K(\bP z_1)z_1& = 2\1\omega \left({ \alpha c(t) K_c^+\over 2}\right)^2
\wlim {w_{max}^2\over (\lambda -\1\omega)^2}  \\
&+ K\left(\alpha c(t) w_{max}+O(\alpha^2)\right)\left(\alpha c(t) v^+_c  +O(\alpha^2)\right) +O(\alpha^3)\\
&=-\alpha^2 c(t)^2 {(K_c^+)^2\over 2}
\wlim { (\lambda -\1\omega)-\lambda \over (\lambda -\1\omega)^2}w_{max}^2 \\
&+ K_c^+\alpha^2 c(t)^2 w_{max} {K_c^+\over 2} \wlim {w_{max}\over \lambda -\1\omega} 
+O(\alpha^3)\\
&=-\alpha^2 c(t)^2 {(K_c^+)^2\over 2} 
\wlim {w_{max}^2 \over \lambda -\1\omega} +\alpha^2 c(t)^2 {(K_c^+)^2\over 2} 
\wlim {w_{max}^2 \over \lambda -\1\omega} +O(\alpha^3)\\
&=O(\alpha^3).
\end{split}
\end{equation*}

\end{proof}


\subsection{The slow manifold reduction}

Projecting both sides of \eqref{dZ1-2} onto the center subspace, we have
\be\lbl{project-center}
\bPi_0 \dot z_1= \bPi_0 \bT_0 z_1 +{\epsilon\over 2} \bPi_0 h -
{K\over 2} \bPi_0 (\overline{h} z_2).
\ee
Using \eqref{AN}, we have
\be\lbl{collect}
\begin{split}
\bPi_0 \dot z_1 & =\alpha \dot c(t) v_c^+,\\
\bPi_0 \bT_0 z_1 &= \bT_0^\times \bPi_0 z_1=\alpha c(t) \bT_0^\times v_c^+ =0.
\end{split}
\ee

Further,
\be\lbl{Pih}
\bPi_0 h = g_1\wlim (\lambda-\1\omega)^{-1} \tilde\bPi_{\mu_{max}}
\bD(\lambda) h.
\ee
To evaluate \eqref{Pih}, we take the following steps 
\begin{equation*}
\begin{split}
\wlim \bD(\lambda) h
&= \wlim \int_\R \frac{\alpha c(t) w_{max}}{\lambda -i\omega }g(\omega )d\omega +O(\alpha ^2) \\
&= \alpha c(t) w_{max} \mathcal{D}(0+) +O(\alpha^2) \\
&= {2\alpha c(t) w_{max}\over K_c^+ \mu_{max}} +O(\alpha^2)
\end{split}
\end{equation*}
and 
$$
\l \tilde\bPi_{\mu_{max}} \bD(\lambda) h
= {2\alpha c(t) w_{max}\over K_c^+ \mu_{max}}+O(\alpha^2).
$$

Finally, 
\be\lbl{Pih-1}
\begin{split}
\bPi_0 h &= \frac{2\alpha  c(t) }{K_c^+ \mu_{max}}g_1\wlim { w_{max} \over \lambda -\1\omega}+
O(\alpha^2)\\
&=\frac{\alpha c(t)}{\mu_{max}} g_1 \left({2\over K_c^+}\right)^2 v_c^+ +O(\alpha^2). 
\end{split}
\ee

Similarly, to evaluate  
\be\lbl{Pihz2}
\bPi_0 (\overline{h} z_2) =g_1 \wlim (\lambda-\1\omega)^{-1} 
\tilde \bPi_{\mu_{max}} \bD (\lambda) (\overline{h} z_2),
\ee
we first compute
\be\lbl{Pihz2-by-step}
\begin{split}
\wlim \bD(\lambda) (\overline{h} z_2) &= \alpha^3 |c(t)|^2 c(t)
\left({K_c^+\over 2}\right)^2
|w_{max}|^2 w_{max} \l \int_\R {g(\omega)d\omega\over (\lambda -\1\omega)^3} + O(\alpha ^4)\\
&=-\alpha^3 |c(t)|^2 c(t) g_2 \left({K_c^+\over 2}\right)^2
|w_{max}|^2 w_{max} +O(\alpha^4),
\end{split}
\ee
where
\be\lbl{def-g2}
g_2= {\pi g''(0)\over 2}.
\ee
By plugging \eqref{Pihz2-by-step} into \eqref{Pihz2}, we obtain
\be\lbl{Pihz2-1}
\bPi_0 (\overline{h} z_2) =- \alpha^3 |c(t)|^2 c(t)  g_1 g_2 \frac{K_c^+}{2} 
\frac{\tilde \bPi_{\mu_{max}} (|w_{max}|^2 w_{max})}{w_{max}}v_c^+ +O(\alpha^4).
\ee

By plugging \eqref{collect}, \eqref{Pih-1}, and \eqref{Pihz2-1} into \eqref{project-center},
dividing both sides by $\alpha$ and $v_c^+$ and keeping terms up to $O(\alpha^2)$ we have
\be\lbl{slow}
\dot c =  {2 g_1 c\over (K_c^+)^2\mu_{max}}
\left( \epsilon + \frac{(K_c^+)^4 \mu_{max}g_2}{8}\alpha^2 |c|^2 \frac{\tilde \bPi_{\mu_{max}} (|w_{max}|^2 w_{max})}{w_{max}}\right) 
+ O(\alpha ^3).
\ee
It is instructive to recast \eqref{slow} in terms of the order parameter $h$ (cf.~\eqref{corder}).
By Lemma~\ref{lem.order},
$$
h(t,x)=\alpha c(t)w_{max} + O(\alpha^2).
$$
Thus, by multiplying both sides of \eqref{slow} by $\alpha w_{max}$ and neglecting higher
order terms, we  obtain
\be\lbl{slow-h}
\begin{split}
\dot h &=  {2 g_1 \over (K_c^+)^2 \mu_{max}} h
\left( \epsilon +\frac{(K_c^+)^4 \mu_{max}g_2}{8}\frac{\tilde \bPi_{\mu_{max}} (|w_{max}|^2 w_{max})}{|w_{max}|^2 w_{max} } |h|^2 \right) \\
&={2 g_1 \over (K_c^+)^2 \mu_{max}} h
\left( \epsilon + \frac{g''(0)}{\pi^3 g(0)^4 \mu_{max}^3} C(x) |h|^2 \right).
\end{split}
\ee
Equation \eqref{slow-h} shows that the trivial solution (the incoherent state) looses stability 
at $\epsilon=0$ and for small $\epsilon>0$ there is a nonzero stable equilibrium
\be\lbl{pitchfork}
\begin{split}
\left| h_{\infty}\right|&= \sqrt{\frac{-8}{(K_c^+)^4 \mu_{max}g_2} 
  \frac{|w_{max}|^2 w_{max}}{\tilde \bPi_{\mu_{max}} (|w_{max}|^2 w_{max})}} \cdot \sqrt{K-K_c^+} +o(\sqrt{K-K_c^+}) \\
&={ g(0)^2 \pi^{3/2} \over \sqrt{-g^{''}(0) }} \mu_{max}^{3/2} 
\sqrt{\frac{1}{C(x)}}\sqrt{K-K_c^+} +o(\sqrt{K-K_c^+}), \quad K>K_c^+.
\end{split}
\ee

\subsection{Examples}
In \cite{ChiMed16}, we derived the transition formulas for the 
onset of synchronization in the KM on several networks. We now return to these examples and 
describe the transition to synchronization in more detail using the results of this 
section.

We start with the KM on the Erd\H{o}s-R{\' e}nyi graphs. To this end,
let  $W\equiv p \in (0,1)$. In \cite{ChiMed16}, we showed that the largest positive
eigenvalue of $\mathbf{W}$ in this case is $\mu_{max}=p$. 
The corresponding eigenfunction $w_{max}$ is constant.  This yields the critical
value $K_c^+=2(\pi g(0)p)^{-1}$. By plugging in these values into \eqref{fork},
we obtain
\be\lbl{fork-ER}
h_\infty (K)= { g(0)^2 \pi^{3/2} \over \sqrt{-g^{''}(0) }} p^{3/2} \sqrt{K-K_c^+}
+ o(\sqrt{K-K_c^+}).
\ee

We next turn to the KM on small-world graphs. This family of graphs is defined via
the following graphon:
\be\lbl{def-Wpr}
W_{p,r}(x,y)=
\left\{ 
\begin{array}{ll}
1-p, & \min\{|x-y|,\, 1-|x-y|\} \leq r,\\
p, &\mbox{otherwise},
\end{array}
\right.
\ee
where $p,r\in (0,1/2)$ are two parameters. The former stands for the probability 
of long range random connections and the latter is the range of regular local connections
(cf.~\cite{Med14c}). 

The largest eigenvalue of $\mathbf{W}_{p,r}$ is equal to $2r+2p-4pr$ 
and the corresponding eigenfunction $w_{max}$ is constant (cf.~\cite{ChiMed16}).
This implies that the critical value is 
$$
K_c^+={2\over \pi g(0) ( 2r +p -4pr)}.
$$
Using \eqref{fork}, we further have
\be\lbl{fork-SW}
h_\infty(K)= { g(0)^2 \pi^{3/2} \over \sqrt{-g^{''}(0) }} (2r+2p-4pr)^{3/2}
\sqrt{K-K_c^+}
+ o(\sqrt{K-K_c^+}).
\ee

\section{Bifurcation with a two-dimensional null space}\lbl{sec.bif2}
\setcounter{equation}{0}

\subsection{The slow manifold reduction}
Many networks in applications can be described with the limiting graphon of the following form
\be\lbl{Wsymmetry}
W(x,y)=G(x-y)
\ee 
for some $G\in L^2(\SS)$ such that $G(x) = G(-x)$. The graphons of this form are used in the description of the small-world
and many other networks (cf.~\S 5.3 \cite{ChiMed16}).

A graphon satisfying \eqref{Wsymmetry} admits Fourier series expansion
\be\lbl{Fseries}
W(x,y)=\sum_{k\in\Z} c_k e^{2\pi\1 k(x-y)}, \; c_{-k}=c_k\in \R.
\ee
By Parseval's identity,
\be\lbl{Parseval}
\sum_{k\in \Z} c_k^2= (2\pi)^{-1} \|G\|^2_{L^2(\SS)}<\infty.
\ee

It follows from \eqref{Fseries} that the eigenvalues of the kernel operator 
$\bW$ coincide with the Fourier coefficients
$c_k, \; k\in\Z$. The Fourier modes $e^{2\pi\1 kx}, \; k\in \Z,$ 
yield the corresponding eigenfunctions.

We continue to assume that the largest eigenvalue of $\bW$ is positive, i.e.,  there is at least 
one positive coefficient $c_k$. In view of \eqref{Parseval}, there is 
a finite set 
\be\lbl{maximal}
M=\left\{ m\in\Z:\;c_m=\sup\{ c_k:\; k\in\Z\}\right\}.
\ee
If $M=\{0\}$ the null space of $\bW$ is one-dimensional. This case was analyzed in the previous section.
Here, we assume $|M|=2$, i.e., there exists a unique $m\in \N$ such that
$$
\mu_{max} = \sup\{ c_k:\; k\in\Z\} = c_m= c_{-m}.
$$
The corresponding eigenspace is spanned by $w_+= e^{2\pi\1 mx}$ and $w_-= e^{-2\pi\1 mx}.$

From now on, the slow manifold reduction proceeds along the lines of the analysis in Section~\ref{sec.bif}.
The generalized center subspace of $\bT_0$ is spanned by
\be\lbl{2center}
v_\pm ={K_c^+\over 2} \wlim {w_\pm (x)\over \lambda -\1 \omega}.
\ee
On the center manifold, we adopt the following Ansatz
\begin{eqnarray}\lbl{2AN}
z_1 &=& \alpha \left( c_-(t) v_- + c_+(t) v_+\right) + O(\alpha^2),\\
z_k &=& h_k(z_1) = O(\alpha^2), \quad k=2,3,\dots, \\
\epsilon &=& \alpha^2, 
\end{eqnarray}
where $(c_-, c_+)$ is the coordinate along the center manifold.
Following the lines of Lemma~\ref{lem.order} and \ref{lem.z2}, we obtain
\be\lbl{2h}
h(t,x)= \alpha\left( c_-(t)w_-(x)+c_+(t)w_+(x)\right)+O(\alpha^2)
\ee
and 
\be\lbl{2z_2}
z_2 =\left[{\alpha K_c^+\over 2} \left(  c_-(t)w_-(x)+c_+(t)w_+(x)\right)\right]^2 
         \cdot \wlim {1\over (\lambda -\1\omega)^2}+ O(\alpha^3).
\ee

In analogy to \eqref{project-center} and \eqref{collect},  projection of
 \eqref{dZ1-2} onto the center subspace yields
\be\lbl{2proj-center}
\alpha ( \dot c_-v_- +\dot c_+v_+ ) = {\epsilon\over 2} \bPi_0 h -
{K\over 2} \bPi_0 (\overline{h} z_2).
\ee
As in \eqref{Pih-1} and \eqref{Pihz2-1}, we further obtain
\be\lbl{2Pi0h}
\bPi_0 h = {\alpha g_1\over \mu_{max}} \left({2\over K^+_c}\right)^{2} \left( c_-v_- + c_+v_+\right)
+O(\alpha^2).
\ee
and
\be\lbl{2Pi0barh}
\bPi_0(\bar h z_2)=-\alpha ^3 g_1g_2 \left( \frac{K_c^+}{2}\right)^2 \cdot \wlim \frac{1}{\lambda -i\omega } 
\tilde\bPi_{\mu_{max}} ((c_-w_-+c_+w_+)^2 (\overline{c_-w_-+c_+w_+})) + O(\alpha^4).
\ee
Taking into account that $w_\pm=e^{\pm 2\pi \1 mx}$, we compute
\be\lbl{2Pi0}
\tilde\bPi_{\mu_{max}}
((c_-w_-+c_+w_+)^2 (\overline{c_-w_-+c_+w_+}))\\
=\left( c_-|c_-|^2 + 2c_-|c_+|^2\right) w_-
+\left( c_+|c_+|^2 + 2c_+|c_-|^2\right) w_+.
\ee

Plugging \eqref{2Pi0} into \eqref{2Pi0barh}, we obtain
\be\lbl{2Pi0barh=}
\bPi_0(\bar h z_2) = -{\alpha^3 g_1g_2 K_c^+\over 2}
\left(\left( c_-|c_-|^2 + 2c_-|c_+|^2 \right) v_-
+\left( c_+|c_+|^2 + 2c_+|c_-|^2\right) v_+  \right)
+O(\alpha^4).
\ee

Combining \eqref{2proj-center}, \eqref{2Pi0h}, and \eqref{2Pi0barh=}, and by comparing
the coefficients of $v_\pm$ on both sides of the resultant equation, we arrive at 
the following system of equations 
\be\lbl{2dreduction}
\left\{ \begin{array}{ll} 
\dot c_- &= p_1 c_-\left(\epsilon - \alpha ^2p_2 \left( |c_-|^2+2 |c_+|^2\right)\right) + O(\alpha^3),\\
\dot c_+ &= p_1 c_+\left(\epsilon - \alpha ^2p_2 \left( |c_+|^2+2 |c_-|^2\right)\right) + O(\alpha^3),
\end{array}
\right.
\ee
where 
\be\lbl{p1-p2}
p_1 ={ 2g_1 \over (K_c^+)^2 \mu_{max}}, \quad
p_2 = -2g_2 \left( {K_c^+\over 2} \right)^4 \mu_{max}
\ee
are positive constants (see Theorem \ref{thm.Riesz}).

Using the polar form for $c_\pm =r_{\pm} e^{\1\phi_\pm}$, we rewrite  
\eqref{2dreduction}
as follows
\be\lbl{r-phi-reduction}
\left\{ \begin{array}{ll} 
\dot r_- &=  p_1r_-\left(\epsilon - \alpha ^2p_2 \left( r_-^2+2 r_+^2\right)\right) + O(\alpha^3),\\
\dot r_+ &=  p_1r_+\left(\epsilon - \alpha ^2p_2 \left( r_+^2+2 r_-^2\right)\right) + O(\alpha^3),\\
\dot \phi_- &=0,\\
\dot \phi_+ &=0.
\end{array}
\right.
\ee
Neglecting the higher order terms, we locate the fixed points
$$
(r_-,r_+)= (0,0), \;  (0, \sqrt{\epsilon/p_2\alpha ^2}), \; (\sqrt{\epsilon/p_2\alpha ^2}, 0), \; 
(\sqrt{\epsilon/3p_2\alpha ^2},\sqrt{\epsilon/3p_2\alpha ^2}).
$$
The linearization of \eqref{r-phi-reduction} about these fixed points yields
$$
\epsilon p_1 \begin{pmatrix}
1 & 0 \\ 0 & 1
\end{pmatrix}, \; 
\epsilon p_1 \begin{pmatrix}
-1& 0  \\ 0& -2 
\end{pmatrix}, \; 
\epsilon p_1
\begin{pmatrix}
-2& 0 \\ 
0& -1 
\end{pmatrix}, \; 
{2\over 3}\epsilon p_1 
\begin{pmatrix}
-1 & -2 \\ 
-2 & -1 
\end{pmatrix}, 
$$
respectively. Thus, the second and the third fixed points are stable for $0<K-K_c^+\ll 1$.
This proves that the order parameter (\ref{2h}) tends to
\be\lbl{hinfty+}
h^+_\infty(x) = \sqrt{\frac{K-K_c^+}{p_2}}e^{i\phi} w_+(x) + o(\sqrt{K-K^+_c})
\ee
or
\be\lbl{hinfty-}
h^-_\infty(x) = \sqrt{\frac{K-K_c^+}{p_2}}e^{i\phi} w_-(x) + o(\sqrt{K-K^+_c})
\ee
as $t\to \infty$, where $\phi$ is a constant which depends on an initial condition.

\begin{figure}
\begin{center}
\textbf{a}~\includegraphics[height=1.8in,width=2.0in]{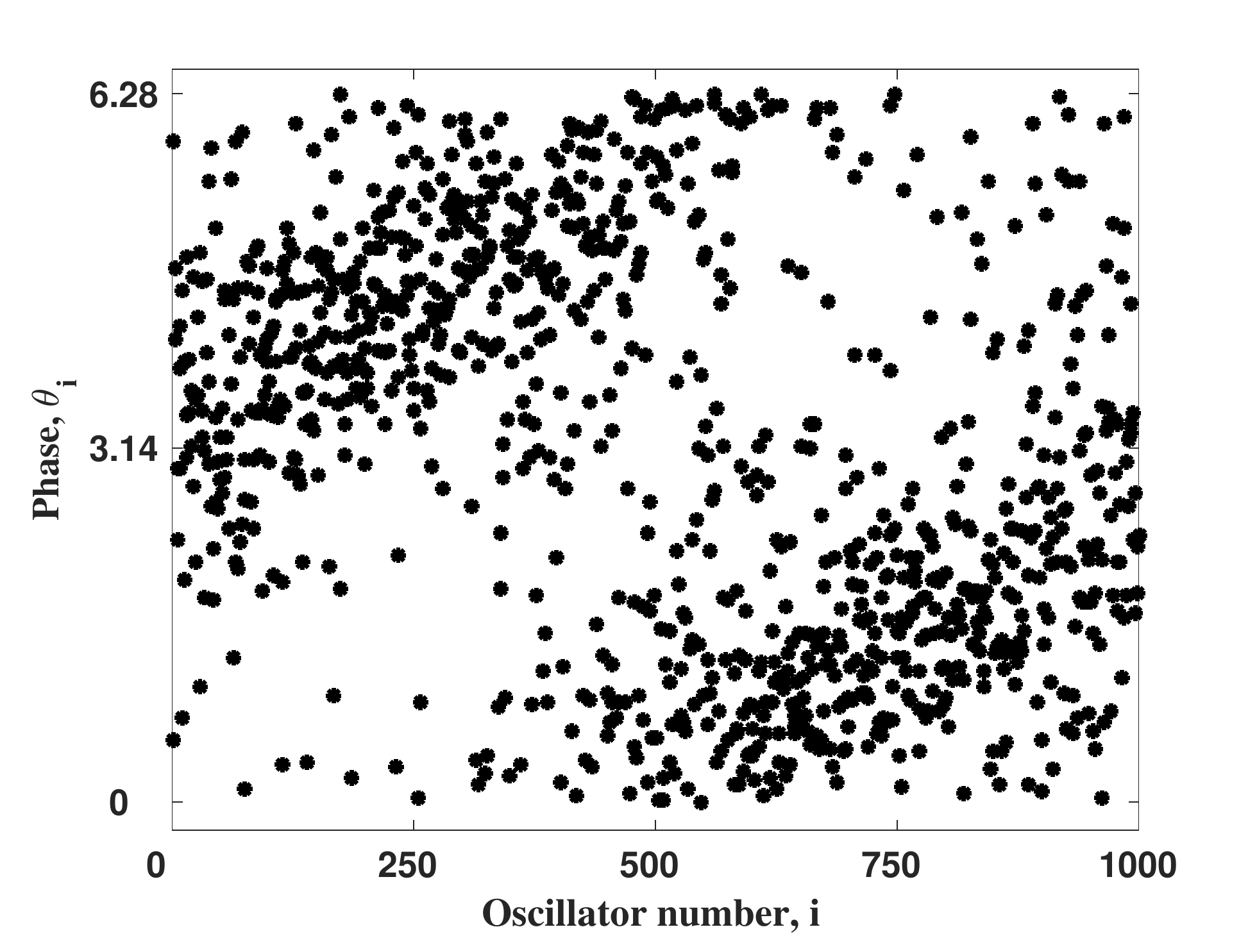}
\textbf{b}~\includegraphics[height=1.8in,width=2.0in]{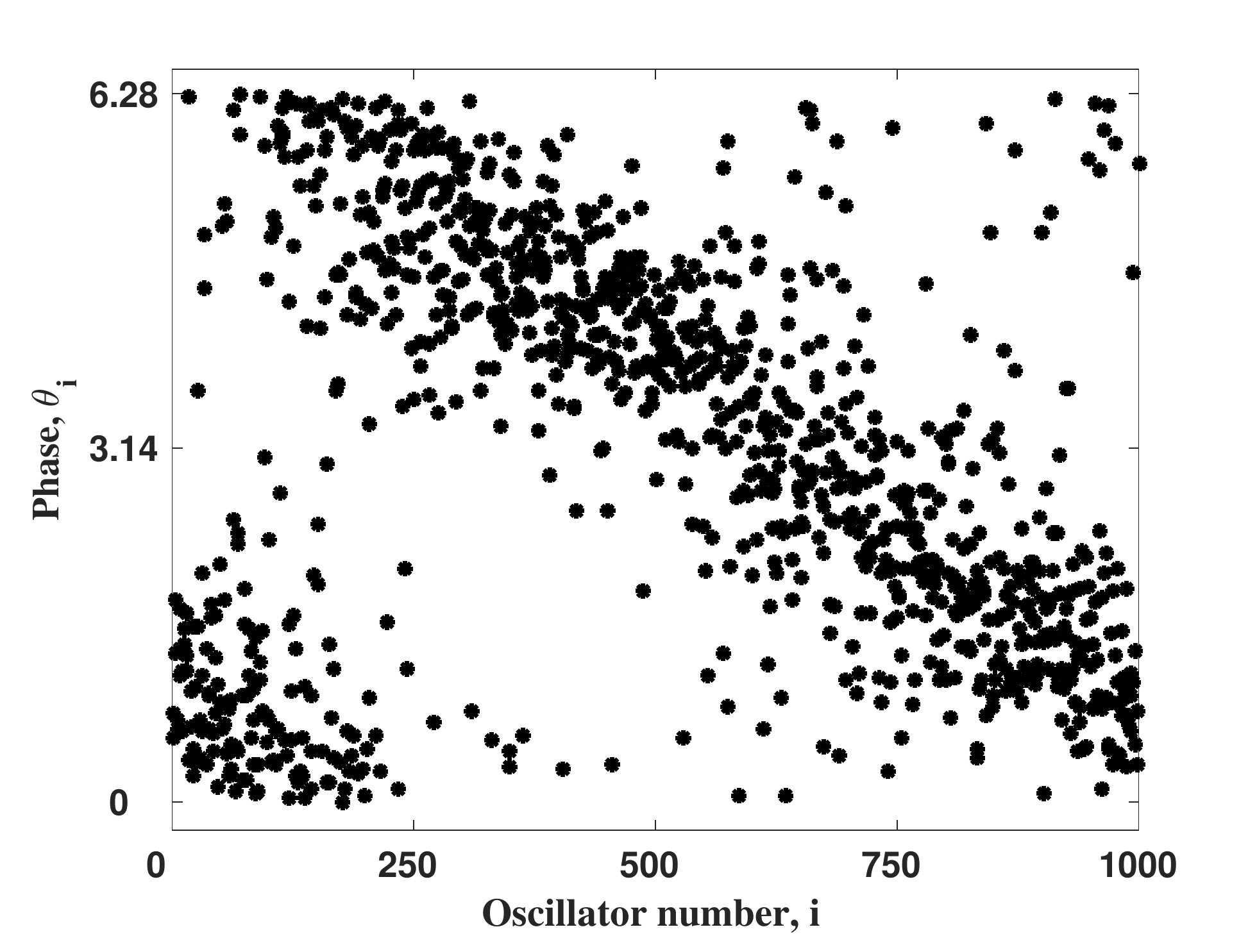}
\textbf{c}~\includegraphics[height=1.8in,width=2.0in]{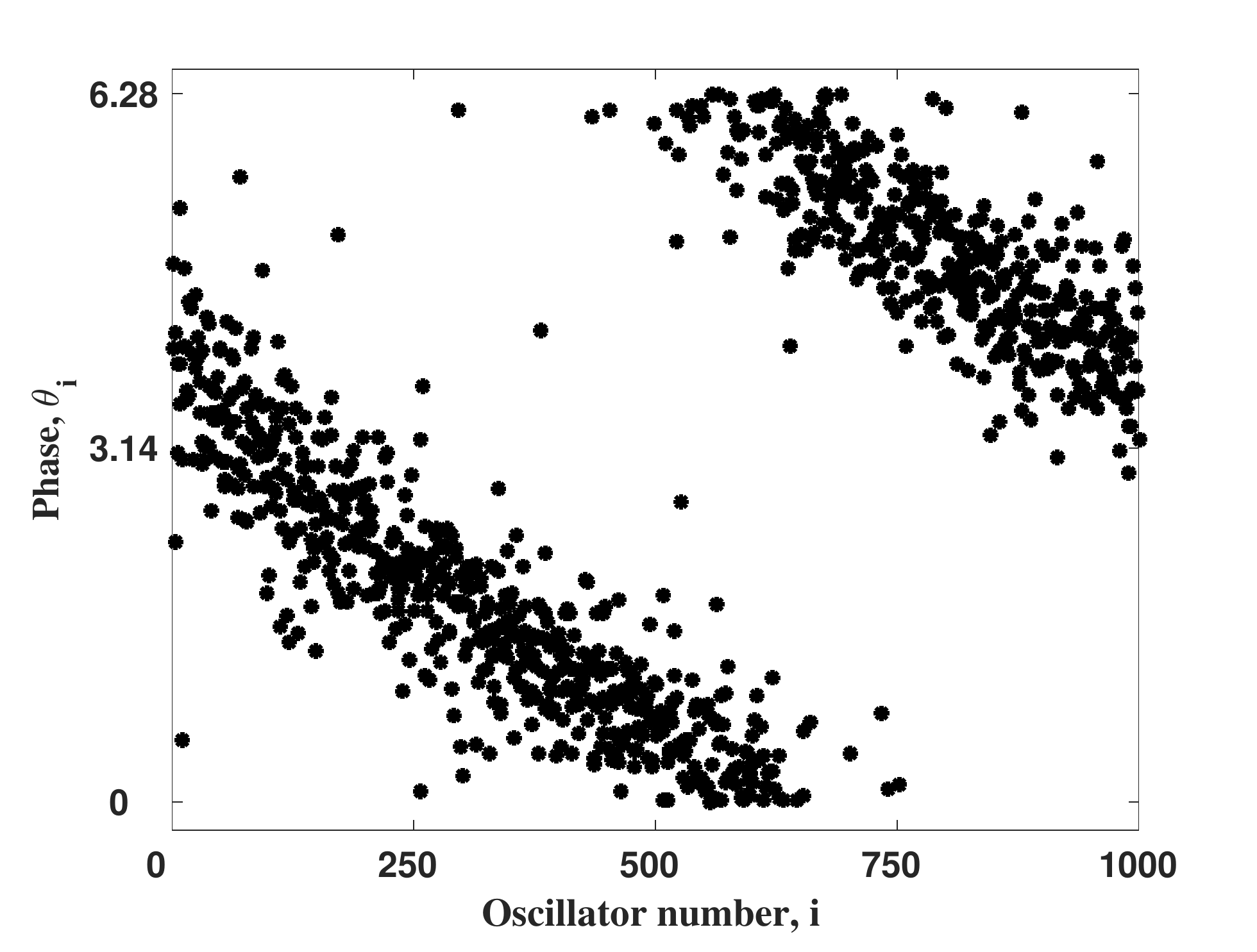}
\end{center}
\caption{Formation of partially phase-locked solutions near a bifurcation with two-dimensional
null space. The KM with intrinsic frequencies from the standard normal distribution, 
graphon \eqref{cos-kernel}, and random initial condition  was for suffiently large time
to reach a stationary regime.
The values of $K$ are  \textbf{a}) $3.5$, \textbf{b}) $4$, and \textbf{c}) $5$.  
The asymptotic state in (\textbf{a}) combines oscillators grouped around a $1$-twisted 
state with those distributed randomly around $\SS$. For increasing values of $K$, the 
noisy twisted states become more distinct (\textbf{b}, \textbf{c}).
}\lbl{f.3}
\end{figure}

\subsection{Example}

To illustrate the bifurcation with two-dimensional null space, let
\be\lbl{cos-kernel}
W(x,y))=\cos\left( 2\pi (x-y)\right)={1\over 2} e^{-2\pi\1 (x-y)} +{1\over 2} e^{2\pi\1 (x-y)}.
\ee
The only eigenvalue of $\bW$ is $\mu=0.5$ and the corresponding
 eigenfunctions are
$$
w_\pm(x)=e^{\pm2\pi\1 x}.
$$

The analysis of this section then yields two stable branches of solutions bifurcating 
from $h\equiv 0$ at $K_c^+=4(\pi g(0))^{-1}\approx 3.2$ (cf.~\eqref{hinfty+}, \eqref{hinfty-}):
$$
h_\infty (x) = \sqrt{\frac{\kappa}{p_2}} e^{\1(\pm 2\pi x+\phi)} + o(\sqrt{\kappa}), \quad
0<\kappa=K-K^+_c\ll 1, \;p_2 = -\frac{8g''(0)}{\pi^3 g(0)^4},
$$
where the phase shift $\phi$ is determined from the initial condition.
For small $\kappa>0$, the system has a family of stable partially phase-locked solutions
\eqref{part-sync}, which can be described as follows. The oscillators split into two groups
depending on their intrinsic frequencies. If $|\omega_i| < K\sqrt{\frac{\kappa}{p_2}},$
the oscillator $i$ approaches one of the two phase-locked solutions:
\be\lbl{ntwist}
\theta_i= {\pm 2\pi i\over n} + \phi + Y(\omega_i)+o(\sqrt{\kappa}),
\ee
where $Y(\omega_i)=\arcsin\left(\frac{\omega_i }{K}\sqrt{\frac{p_2}{\kappa}}\right)
\in (-\pi/2, \pi/2)$ is a function of the random intrinsic frequency $\omega_i$. 
The oscillators in this group form a noisy $1$-twisted state \cite{MedTan15b}. 
The oscillators with intrinsic frequencies 
$|\omega| > K\sqrt{\frac{\kappa}{p_2}}$ are randomly distributed around $\SS$.
The density of this distribution is given in the second line of \eqref{part-sync}.

Figure~\ref{f.3} presents results of numerical integration of the KM with graphon
\eqref{cos-kernel} and randomly distributed intrinsic frequencies. The plots
in Figure~\ref{f.3} \textbf{a}-\textbf{c} show asymptotic states of the KM for
three increasing values of $K,$ starting with $K=3.5$  just near the 
critical value $K_c^+\approx 3.2$. In Figure~\ref{f.3}\textbf{a} there are many oscillators
spread around $\SS$. However, the group of oscillators concentrating about a $1$-twisted 
state is already visible. For larger values of $\kappa$, the twisted state becomes more
pronounced (see Figure~\ref{f.3}\textbf{b},\textbf{c}). Twisted states bifurcating 
from the incoherent state are also present in the KM on small-world graphs 
(see \cite{ChiMedMiz18} for the analysis of the small-world network and other examples).

\vskip 0.2cm
\noindent
{\bf Acknowledgements.}
This work was supported in part by the NSF  DMS  grant 1715161 (to GM).


\def\cprime{$'$} \def\cprime{$'$} \def\cprime{$'$}
\providecommand{\bysame}{\leavevmode\hbox to3em{\hrulefill}\thinspace}
\providecommand{\MR}{\relax\ifhmode\unskip\space\fi MR }
\providecommand{\MRhref}[2]{%
  \href{http://www.ams.org/mathscinet-getitem?mr=#1}{#2}
}
\providecommand{\href}[2]{#2}

\end{document}